\documentclass[11pt]{article}

\usepackage[utf8]{inputenc}
\usepackage{fullpage}
\usepackage{bm}
\usepackage[dvipsnames]{xcolor}
\usepackage{amsmath,amssymb}
\usepackage{amsthm}
\usepackage{wrapfig}
\usepackage{epsfig}
\usepackage{graphicx}
\usepackage{makeidx}
\usepackage{multicol}
\usepackage{tikz}
\usepackage{algorithm}
\usepackage{algpseudocode}
\usepackage{soul}

\usepackage{array,multirow,booktabs,bbm} 

\usepackage{algorithmicx}
\usepackage{algpseudocode}
\usepackage{epstopdf}
\usepackage{mathrsfs,mathtools}
\usepackage{xcolor}
\usepackage{enumitem}
\usepackage{comment}
\usepackage{hyperref}

\newcommand{\reff}[1]{{\rm (\ref{#1})}}

\graphicspath{{./Figs/}}
\newcommand{\ve}{\varepsilon}          

\newcommand{\calL}{\mathcal{L}}

\newtheorem{theorem}{Theorem}[section]

\newtheorem{lemma}{Lemma}[section]
\theoremstyle{remark}
\newtheorem{remark}{Remark}[section]

\begin{document}
\graphicspath{{Figs/}}

\title{
Asymptotic analysis on charging dynamics for stack-electrode model of supercapacitors}

\author{
Lijie Ji\thanks{School of Mathematical Sciences, Shanghai Jiao Tong University, Shanghai 200240, China. Email: {\tt{sjtujidreamer@sjtu.edu.cn}}},~
\and
Zhenli Xu\thanks{School of Mathematical Sciences, MOE-LSC, CMA-Shanghai, and Shanghai Center for Applied Mathematics, Shanghai Jiao Tong University, Shanghai 200240, China. Email: {\tt{xuzl@sjtu.edu.cn}}},~
\and
Shenggao Zhou\thanks{
School of Mathematical Sciences, MOE-LSC, CMA-Shanghai, and Shanghai Center for Applied Mathematics, Shanghai Jiao Tong University, Shanghai 200240, China. Email: {\tt{sgzhou@sjtu.edu.cn}}.} 
}

\date{\empty}

\maketitle
\begin{abstract}
Supercapacitors are promising electrochemical energy storage devices due to their prominent performance in rapid charging/discharging rates, long cycle life, stability, etc. Experimental measurement and theoretical prediction on charging timescale for supercapacitors often have large difference.  
This work develops a matched asymptotic expansion method to derive the charging dynamics of supercapacitors with porous electrodes, in which the supercapacitors are described by the stack-electrode model. Coupling leading-order solutions between every two stacks by continuity of ionic concentration and fluxes leads to an ODE system, which is a generalized equivalent circuit model for zeta potentials, with the potential-dependent nonlinear capacitance and resistance determined by physical parameters of electrolytes, e.g., specific counterion valences for asymmetric electrolytes. Linearized stability analysis on the ODE system after projection is developed to theoretically characterize the charging timescale. The derived asymptotic solutions are numerically verified. Further numerical investigations on the biexponential charging timescales demonstrate that the proposed generalized equivalent circuit model, as well as companion linearized stability analysis, can faithfully capture the charging dynamics of symmetric/asymmetric electrolytes in supercapacitors with porous electrodes.

\end{abstract}


\section{Introduction}

Supercapacitors, also known as electric double layer (EDL) capacitors, store ions in electric double layers near the electrode/electrolyte interface. Compared to planar electrodes, porous electrodes that accommodate dense packing of ions in pores are more often considered in supercapacitors, to enlarge the area of the interface per unit volume~\cite{xiong2019recent,largeot2008relation,
bi2020molecular,porada2013review,
biesheuvel2011theory,ortiz2022understanding}.  Supercapacitors have attracted considerable attention in recent years due to their excellent performance in rapid charging/discharging rates, long cycle life, stability, and high power density~\cite{chmiola2006anomalous,limmer2013charge,forse2017direct, prehal2018salt, niu2022conductive,choudhary2017asymmetric}. Such attractive features make them promising electrical energy storage devices with wide applications in regenerative braking, smart grids, and electric vehicles~\cite{pilon_jps14,Pilon_JPS_15, lian2020blessing, HuangLianLiu_AICHEJ22}.

Molecular dynamics simulations~\cite{KondratKornyshev_NatureMat14}, lattice Boltzmann simulations~\cite{AstaRotenberg_JCP19,basu2020fully}, mean-field theories, the Poisson--Nernst--Planck (PNP) theory~\cite{BTA:PRE:04,berg2011analytical}, etc., have been applied to understand the underlying molecular mechanism. For instance, Groda \textit{et al.} have mapped the theory of charging supercapacitors on the Bethe-lattice model to understand generic features of energy storage, e.g., the ion ordering inside a pore \cite{groda2021superionic}.
Bazant \textit{et al.} have analyzed the charging dynamics between two parallel electrodes systematically in the work \cite{BTA:PRE:04}, in which a time-dependent matched asymptotic expansion (MAE) method has been proposed for boundary-layer analysis. Other singular perturbation analysis has also been used to systematically understand the charge dynamics described by the PNP theory~\cite{WLiu_SIAP05}, investigate the charge transport and the structure of electric double layers with inhomogeneous dielectric permittivity~\cite{ji2018asymptotic}, and derive the current-voltage relations of electrochemical systems \cite{BazantChuBayly_SIAP06,ChuBazant:SIAP:2005,abaid2008asymptotic}. The classical transmission line (TL) model~\cite{de1963porous}  based on linearized electrolyte theories has been proposed in literature to describe charging dynamics. By fitting against experimental data, the TL model can successfully predict charging dynamics for general morphologies at low applied potentials~\cite{mirzadeh2014enhanced}.  A modified model that assumes a volume-averaged,  electroneutral bulk has been developed by Biesheuvel and Bazant (BB) to consider charging process at higher potentials, and predicts that the charging dynamics slows down due to nonlinear capacitance and salt depletion of the pore~\cite{biesheuvel2010nonlinear}. Direct numerical simulations of the full PNP equations in simple 2D porous structures reveal that the TL model underestimates and the BB model overestimates the charging timescale~\cite{mirzadeh2014enhanced}. To improve accuracy, a surface conduction mechanism, that can effectively short circuit of the high-resistance electrolyte in the bulk of the pores, has been further proposed based on the BB model.

Due to limitation of computational capacity, most of the aforementioned studies focused on nanoscale systems with simple geometry, such as parallel planar electrodes or coaxial cylindrical electrodes~\cite{BTA:PRE:04, AstaRotenberg_JCP19}. The predicted results on charging timescales are several orders of magnitude smaller than macroscopic experimental data, on account of the ignorance of the multiscale nature and porous structure in electrodes.  Although effective medium approximation or empirical fitting approaches could be used to bridge the gaps in scales, multiscale models that can faithfully capture the ion dynamics in charging processes of supercapacitors are still highly desirable. Recently, a multi-stack model has been proposed to mimic the porous structure of electrodes in supercapacitors~\cite{lian2020blessing}. Stacks that allow ionic penetration are 
parallelly placed in electrodes with applied voltages, so that the system can be reduced to one dimension. Classical equivalent circuit models and numerical simulations based on the PNP theory are employed to predict the charging timescale of electrodes with porous structures. Timescale analysis based on an ODE system for an equivalent circuit model with constant capacitance and resistance demonstrates that the timescale of charging increases linearly with the number of stacks, which represents the porosity of electrodes. Such a multiscale model that bridges the timescale gap between theories and experiments has been employed to study heat generation of porous electrodes, etc.~\cite{LianLiuRoij_PRL22, HuangLianLiu_AICHEJ22, TaoLianLiu_AICHEJ22}. {Aslyamov and Janssen have analytically derived the biexponential charging dynamics of a long electrolyte-filled slit pore in \cite{aslyamov2022analytical} and these analytical results are consistent with the experiment data in \cite{lian2020blessing}.  Henrique \textit{et al.} have  investigated the impact of asymmetric valences, diffusivities and arbitrary pore size on the charging dynamics inside a charged cylindrical pore. Their perturbation analysis is validated by the direct numerical simulations \cite{GuptaZukStone_PRL2020, henrique2022impact, henrique2022charging}. Further, they derived the MAE model of the electrochemical system with redox reactions in thin-double-layer limit \cite{jarvey2022ion}.   

In this work, we develop an MAE method to analyze the biexponential charging timescales 
for two-ion supercapacitors with the stack-electrode model. By extending the time-dependent matching approach~\cite{BTA:PRE:04}, we obtain an accurate leading-order approximation of the ion concentrations, potential distributions, and total diffuse charge in each region between stacks. By the continuity of concentration and ionic fluxes, we couple the leading-order approximations together to form an ODE system for the zeta potentials at all stacks.
It is of interest to note that the ODE system can be regarded as a generalized equivalent circuit model with potential-dependent nonlinear capacitance and resistance determined by the PNP theory and physical parameters of symmetric/asymmetric electrolytes, such as specific counterion valences. Furthermore, we develop linearized stability analysis on the ODE system after projection, and find the charging timescale by solving an eigenvalue problem. Asymptotic analysis on excess salt concentration indicates the existence of a longer diffusion timescale in late-time charging process.  Numerical simulations are first performed to validate the accuracy of the asymptotic solution. Further numerical studies on the dependence of charging timescale against the number of stacks (or porosity)  demonstrate that the derived generalized equivalent circuit model can faithfully capture the charging process of symmetric and asymmetric electrolyte systems.

The rest of this paper is organized as follows. In Section~\ref{sec:Model}, we review the parallel stack-electrode model and the Poisson-Nernst-Planck theory. In Section~\ref{sec:AsympSolutions}, the MAE method is employed to derive the leading-order approximation for the stack-electrode model. The generalized equivalent circuit model of each zeta potential, effective both for symmetric and asymmetric systems,  is then theoretically derived. Resorting to the linearized stability analysis and the projection technique, we obtain the generalized resistor-capacitor time scale of the supercapacitors with porous electrode in Section~\ref{sec:generalizedRCtime}. Simultaneously, the diffusion timescale can be  estimated from the evolution of salt concentration. In Section~\ref{sec:numerical:results}, the asymptotic solutions are verified by the finite difference method. Further numerical investigations on the biexponential charging timescales for two-ion system demonstrate that the proposed generalized equivalent circuit model through the MAE method can faithfully capture the charging dynamics of symmetric/asymmetric electrolytes. Finally, implications and concluding remarks are made in Section~\ref{sec:conclusion}.

\section{Parallel stack-electrode Poisson-Nernst-Planck model}\label{sec:Model}
Consider a binary electrolyte confined among multiple parallelly stacked planar electrodes of a total width $2{D}$, with $D = \mathcal{L}+H$, in a stack-electrode system~\cite{lian2020blessing}; cf. Fig.~\ref{f:model}. Assume that the system is homogeneous in the $y\hbox{-O-}z$ plane. The valences of the two ion species are denoted by $z_\pm$ with $z_-<0$ and $z_+>0$. 
The mean field electric potential $\Phi$ is determined by the Poisson equation with Dirichlet boundary conditions imposed on each electrode {at $x=\pm x_k$:}
\begin{eqnarray}
\begin{cases}
\displaystyle -\ve_0\nabla \cdot \varepsilon\nabla\Phi=\sum_{i=\pm} z_i e c_i,\\
\displaystyle \Phi(\pm x_k)= V_\pm, ~k=1, \cdots, n.
\end{cases}
\label{p}
\end{eqnarray}
Here $x_k=\mathcal{L}+(k-1)h$, $h=H/(n-1)$, $\ve_0$ is the vacuum permittivity, $\ve$ is the dielectric coefficient, $e$ is the elementary charge, and $c_\pm$ are ionic concentrations.
Under the mean field approximation, the evolution of ionic concentrations are described by the Nernst-Planck (NP) equations

\begin{equation}
\displaystyle \frac{\partial c_i(x,t)}{\partial t}+\nabla \cdot J_i = 0,~~i=\pm,
\label{ns}
\end{equation}
where $J_i=- \mathcal{D}_i [\nabla c_i+\beta c_i \nabla(z_i e \Phi)]$ is the flux density, $\beta=1/k_BT$ {is the inverse thermal energy}, and $\mathcal{D}_{\pm}$ are the diffusion constants of ions. {We} assume $\mathcal{D}_+=\mathcal{D}_-=\mathcal{D}_0$. Combination of~\reff{p} and~\reff{ns} leads to the coupled Poisson--Nernst--Planck (PNP) equations. To mimic porous electrodes, the ions in such a stack-electrode model are assumed to transport freely across all the parallel electrodes except the outmost one. {Therefore, the flux, as assumed in the stack-electrode model~\cite{lian2020blessing}, is set to be continuous across inner electrodes to model ion permeation through porous electrodes, and zero flux is specified at boundaries $x = \pm D$}. When there are only two parallel blocking electrodes, the dynamics and equilibrium properties of this coupled PNP equations have been extensively investigated~\cite{BTA:PRE:04, ji2018asymptotic}. However, such a stack-electrode model that can capture the main physics of porous electrodes has not been mathematically studied in literature. In this paper, the stack-electrode system is investigated by the MAE method, which derives a generalized equivalent circuit model with physical interpretations. Timescale analysis based on the generalized equivalent circuit model is performed to understand the charging dynamics of supercapacitors.
\begin{figure}[H]
\centering
\includegraphics[scale=0.45]{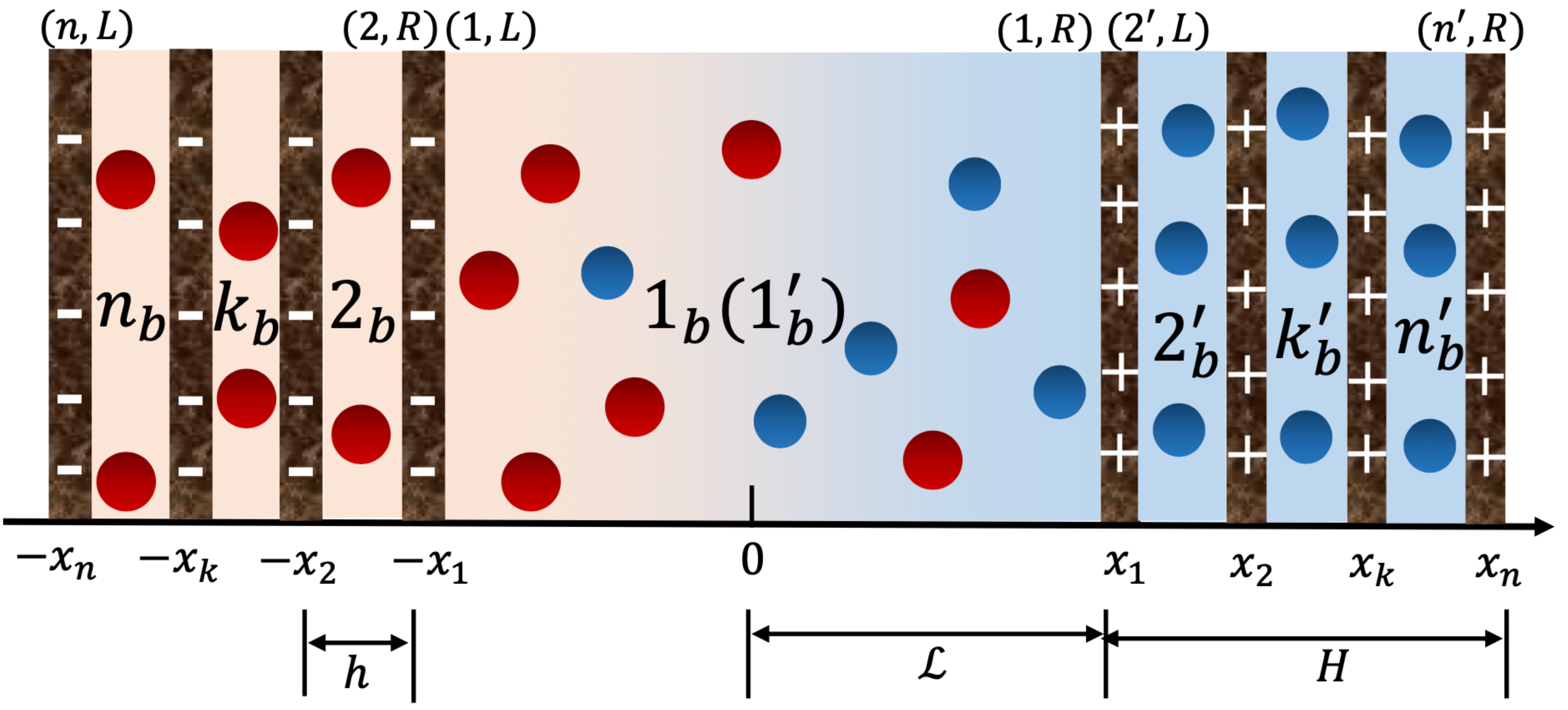}
\caption{Schematic illustration of the {stack-electrode} model system. The porous electrode regions of width $2H$ are described by $2n$ stack electrodes with applied voltages located at $x=\pm x_k, k=1, \cdots, n$ with a uniform spacing $h$. All the left sub-domains are labelled by $1, \cdots, n$ and right sub-domains by $2', \cdots, n'$. More specifically, $(k,L), (k,R), (k',L), (k',R)$, $k=1, \cdots, n$, represent the left EDL and right EDL layers of each left sub-domain and right sub-domain, respectively. The $k_b$ and $k'_b$, $k=1, \cdots, n$, represent the bulk region of each sub-domain. }
\label{f:model}
\end{figure}

We first define some reference variables and obtain the dimensionless form of the model equations. The initial ionic concentrations are assumed to be uniform and charge neutral: $c_\pm(x,0)=c_{\pm,b}=|z_\mp| c_0$ with $c_0$ being a reference concentration. Let $\ell_0= \sqrt{\ve_0\ve/\beta e^2 c_0}$ be a length scale which is proportional to the Debye length $\ell_D =\sqrt{ \ve_0\ve/(\sum_i \beta z_i^2 e^2 c_0)}$.
 Let  $\tau_c = \ell_0D/\mathcal{D}_0$ be the classical resistor-capacitor (RC) time. Introduce dimensionless quantities $\widetilde{x}=x/D$, $\widetilde{t}=t/\tau_c$,
$\widetilde{c}_{\pm}=c_{\pm}/c_0$, and
$\widetilde{\Phi}=\beta e\Phi$.
For ease of presentation, the tildes over all variables in the following discussions are dropped. Then the PNP system {reads}
\begin{eqnarray}
\left\{
\begin{aligned}
&\displaystyle \frac{\partial c_i}{\partial t}=\epsilon \frac{\partial}{\partial x}\left[ \frac{\partial}{\partial x}c_i+ c_i \frac{\partial}{\partial x}\left(z_i\Phi\right)\right],~~i=\pm, ~ -1<x<1, \label{NP1}
   \\
&\displaystyle -\epsilon^2\frac{\partial^2}{\partial x^2}\Phi= \rho,  -1<x<1,
\end{aligned}
\right.
\end{eqnarray}
where $\rho = z_{+}c_{+}+z_{-}c_{-}$. The initial and boundary conditions are
\begin{equation}
\begin{cases}
\displaystyle c_\pm(x,0)=|z_\mp|, \\
\displaystyle \Phi(\pm x_k ,t)=V_\pm,~ k=1, \cdots, n,\\
\displaystyle J_\pm(\pm 1,t)=0.
\end{cases}
\label{boundary}
\end{equation}
Here, $J_{\pm}=\partial_x c_\pm+ c_\pm \partial_x \left(z_\pm\Phi\right)$, is the dimensionless flux. Assuming that the Debye length is much smaller than the separation $D$ and thus $\epsilon=\ell_0/D$ is a small parameter. One shall perform matched asymptotic analysis based on $\epsilon$ to obtain the leading orders of ion concentrations, potential, total diffuse charge, and differential capacitance. Such leading-order solutions are able to provide a generalized equivalent circuit model with physical interpretations. 

\section{Asymptotic solutions}\label{sec:AsympSolutions}
\subsection{The $n=1$ case}
\label{sec:n1}
The method of MAE is first used to solve the initial-boundary value problem of the PNP equations with $n=1$, in which only two electrodes located at $x = \pm 1$ with $H=0$ being considered. 
The outer and inner solutions are first analyzed, and the composite solutions are then constructed through the time-dependent matching method.

\textbf{\textit{Outer solution}}--
When $\epsilon \to 0$, the NP equations become
\begin{equation}
\displaystyle \frac{\partial \overline{c}_{\pm,0}}{\partial t}=0,
\end{equation}
which implies the leading order concentration: $\overline{c}_{\pm,0}(x,t)=\overline{c}_{\pm,0}(x,0)= |z_\mp|$.
Here and afterwards, the subscript ``0" represents the leading term in the asymptotic expansion.
Summing two NP equations in Eq. \eqref{NP1} together and using the Poisson equation, one obtains
\begin{equation}
\displaystyle -\epsilon \frac{\partial^2}{\partial x^2} \frac{\partial}{\partial t} \Phi = -\epsilon^2 \frac{\partial^4}{\partial x^4}\Phi + \frac{\partial}{\partial x}\left[(z_+^2{c}_{+}+z_-^2{c}_{-})\frac{\partial \Phi}{\partial x}\right].
\end{equation}
{Let $\Phi = \overline{\Phi}_0 +\epsilon \overline{\Phi}_1+\cdots$.} Thus, the leading order term $\overline{\Phi}_0$ satisfies
\begin{align}
\displaystyle \frac{\partial}{\partial x}\left[(z_+^2\overline{c}_{+,0}+z_-^2\overline{c}_{-,0})\frac{\partial \overline{\Phi}_0}{\partial x}\right] =0.
\end{align}
Since $\overline{c}_{\pm,0}$ are constants, $\overline{\Phi}_0$ can be written as
$\overline{\Phi}_0(x,t)=\overline{j}_0(t)x+A(t)$, a linear function of $x$. $\overline{j}_0(t)$ is the leading term of the current
density in the outer region with an initial value
$\overline{j}_0(0)=(V_+-V_-)/2$, and $A(0)=(V_-+V_+)/2$. To summarize, the outer solution reads
\begin{equation}
\label{BulkSolution}
\left\{
\begin{aligned}
&\displaystyle \overline{c}_{\pm,0}(x,t)=|z_\mp|, \\
&\displaystyle \overline{\Phi}_0(x,t)=\overline{j}_0(t)x+A(t).
\end{aligned}
\right.
\end{equation}
In order to investigate the dynamics and equilibrium, one shall determine $\overline{j}_0(t)$ and $A(t)$ by matching with the inner solution.

\textbf{\textit{Inner solution}}--
\label{sec:inner:n1}
Due to electrostatic interactions, there is an EDL structure near each electrode. The width of the EDL (inner solution) is of several Debye length, i.e., $O(\epsilon)$ in dimensionless form, and the dynamics in this layer is different from that in the bulk (outer solution). The inner solution near the left interface $x=-1$ is illustrated and the right can be obtained similarly. Denote by $\widehat{c}_\pm(y,t)$ and $\widehat{\Phi}(y,t)$ the inner solutions, and $\widehat{\rho}= z_{+}\widehat{c}_{+}+z_{-}\widehat{c}_{-}$. With a scale transformation $y=(x+1)/\epsilon$, the original PNP equations become
\begin{equation}
\begin{cases}
\displaystyle \epsilon\frac{\displaystyle \partial \widehat{c}_i}{\displaystyle \partial t}=\frac{\displaystyle \partial }{\displaystyle \partial y}\left(\frac{\displaystyle \partial \widehat{c}_i}{\displaystyle \partial y}+\displaystyle z_i\widehat{c}_i \frac{\partial \displaystyle \widehat{\Phi}}{ \displaystyle \partial y}\right), ~i=\pm,\\
\displaystyle -\frac{\displaystyle \partial^2 \widehat{\Phi}}{\displaystyle \partial y^2}=\displaystyle \widehat{\rho}.
\end{cases}
\label{mpnp}
 \end{equation}
Let $\widehat{\Phi}= \widehat{\Phi}_{0} + \epsilon \widehat{\Phi}_{1} +\cdots$, $\widehat{c}_{i}= \widehat{c}_{i,0} + \epsilon \widehat{c}_{i,1} +\cdots$ and $\widehat{\rho}= \widehat{\rho}_{0} + \epsilon \widehat{\rho}_{1} +\cdots$. 
The leading-order terms in \eqref{mpnp} satisfy
\begin{equation}
\frac{\displaystyle \partial \widehat{c}_{i,0}}{\displaystyle \partial y}+\displaystyle z_i\widehat{c}_{i,0}\frac{\displaystyle \partial \widehat{\Phi}_0}{\displaystyle \partial y}=B, \label{second}
\end{equation}
where $B$ is a constant independent of $y$. Using the matching condition of flux densities in both the inner and outer layers, one {gets}
\begin{align}
\displaystyle \lim\limits_{y \rightarrow \infty} \frac{1}{\displaystyle \epsilon}\left[\frac{\displaystyle \partial \widehat{c}_{i,0}}{\displaystyle \partial y}+\displaystyle z_i\widehat{c}_{i,0} \frac{\displaystyle \partial \widehat{\Phi}_0}{ \displaystyle \partial y}+O(\epsilon)\right]
= \displaystyle \lim\limits_{x \rightarrow -1} \left[\frac{\displaystyle \partial \overline{c}_{i,0}}{\displaystyle \partial x}+\displaystyle \overline{c}_{i,0}\frac{\displaystyle \partial}{\displaystyle \partial x}\displaystyle (z_i\overline{\Phi}_0) +O(\epsilon)\right].
\end{align}
Clearly, one has $B=0$. Integrating Eq.~\eqref{second} and matching the ion concentrations and electric potential, i.e., $
\displaystyle \lim\limits_{y\rightarrow \infty} \displaystyle \widehat{c}_{\pm,0}=\lim\limits_{x \rightarrow -1} \displaystyle \overline{c}_{\pm,0},
\lim\limits_{y\rightarrow \infty}\displaystyle  \widehat{\Phi}_0= \displaystyle \lim\limits_{x \rightarrow -1}\displaystyle \overline{\Phi}_0$, 
{one obtains} the following equations
\begin{equation}
\label{IISolution}
\left\{
\begin{aligned}
&\displaystyle \widehat{c}_{\pm,0}=|z_\mp|\exp({-z_\pm \widehat{\varphi}_0}),  \\
&\displaystyle -\frac{\partial^2 \widehat{\varphi}_0}{\partial y^2}= \widehat{\rho}_0,
\end{aligned}
\right.
\end{equation}
where a new variable $\widehat{\varphi}(y,t):=\widehat{\Phi}(y,t)-\overline{\Phi}(-1,t)$ is introduced to describe the potential drop, and $\widehat{\varphi}_0(y,t)$ represents the leading-order term of $\widehat{\varphi}(y,t)$. The boundary conditions for the potential drop are
$\widehat{\varphi}_0(0,t)=\zeta^L(t)$ and $\widehat{\varphi}_0(\infty,t)=0$, where the zeta potential is defined by $\zeta^L(t)  \triangleq \widehat{\Phi}_0(0,t)-\overline{\Phi}_0(-1,t)$. 
Time dependent $A(t)$, $\overline{j}_0(t)$ and zeta potentials can be solved through ODE system \eqref{potentialdropODE1} in Lemma \ref{lemma:1}.

\begin{lemma} \label{lemma:1}
For the $n=1$ case, the zeta potentials satisfy following closed system of ODEs:
\begin{equation}\label{potentialdropODE1}
\begin{cases}
\displaystyle -C (\zeta^L) \frac{d \zeta^L}{dt}=\alpha \frac{V_+ - V_- + \zeta^L - \zeta^R}{2},\\
\displaystyle -C (\zeta^R) \frac{d \zeta^R}{dt}=\alpha \frac{V_- - V_+ + \zeta^R - \zeta^L}{2},
\end{cases}
\end{equation}
with zero initial conditions. Here $\alpha = z_-^2 z_+ -z_+^2 z_-$.

\end{lemma}

\begin{proof}
Consider the leading order asymptotic approximation of the total diffuse charge in the inner layer, $\widehat{Q}_0 =\int_0^{\infty} \left[ z_+\widehat{c}_{+,0}(y,t)+z_-\widehat{c}_{-,0}(y,t)\right] dy$. By Eq.~\eqref{IISolution}, one has
\begin{align}\label{Q0}
\widehat{Q}_0&= \left. \frac{\partial \widehat{\varphi}_0}{\partial y}\right|_{y=0} = \sqrt{2}\sqrt{-z_-e^{-z_+\zeta^L}+z_+e^{-z_-\zeta^L}+z_--z_+},
\end{align}
where integration by parts and $\lim\limits_{y \rightarrow \infty} \partial_y\widehat{\varphi}_0(y,t)= \lim\limits_{y \rightarrow \infty} \partial_y\widehat{\Phi}_0(y,t)=0$ are used in the derivation. Taking the time derivative of the leading-order term of the total diffuse charge and
using the NP equations in Eq.~(\ref{mpnp}), one gets
\begin{align}
\frac{d \widehat{Q}_0}{dt}=\int_0^{\infty} \left( z_{+}\frac{\partial\widehat{ c}_{+,0}}{\partial t}+z_{-}\frac{\partial\widehat{c}_{-,0}}{\partial t} \right) dy
  = \lim_{x\rightarrow -1} \left( \frac{\partial \overline{\rho}_0}{\partial x}+ \alpha \frac{\partial \overline{\Phi}_0}{ \partial x} \right),  \label{diff}
\end{align}
where the zero-flux boundary condition is used. Recall that $\zeta^L(t) = \widehat{\Phi}_0(0,t) - \overline{\Phi}_0(-1,t)= V_- +\overline{j}_0(t)-A(t)$ and the electroneutrality condition $\overline{\rho}_0=0$. Then, the Eq.~\eqref{diff} can be transformed into
\begin{align}
\frac{d \widehat{Q}_0}{d\zeta^L} \frac{d \zeta^L}{dt} =\alpha \overline{j}_0(t),
\label{matching}
\end{align}
where $\overline{j}_0(t)=\partial_x \overline{\Phi}_0$ is independent of $x$. Similarly, for the boundary layer on the right, one can derive an ODE for the potential drop there.
Introduce the differential capacitance $C(\zeta^L)=-d \widehat{Q}_0(\zeta^L)/d \zeta^L$.
By the expression of the total diffuse charge~\reff{Q0}, the exact solution of the differential capacitance can be derived as follows:
\begin{equation}\label{DC}
{C(u) = \text{sgn}(u)\frac{\displaystyle  z_+z_-\left( e^{-z_+ u} -e^{-z_- u}\right)}{\displaystyle \displaystyle \sqrt{2}\sqrt{-z_-e^{-z_+ u}+z_+e^{-z_- u}+z_--z_+}} }.
\end{equation}
It is easy to show that the radicand in~\reff{DC} is nonnegative. Thus, a closed system of ODEs for the potential drops can be derived:
\begin{equation}\label{potentialdropODE}
\begin{cases}
\displaystyle -C (\zeta^L) \frac{d \zeta^L}{dt}=\alpha \frac{V_+ - V_- + \zeta^L - \zeta^R}{2},\\
\displaystyle -C (\zeta^R) \frac{d \zeta^R}{dt}=\alpha \frac{V_- - V_+ + \zeta^R - \zeta^L}{2},
\end{cases}
\end{equation}
where $\zeta^R(t)=V_+-\overline{\Phi}_0(1,t)= V_+ -\overline{j}_0(t)-A(t)$, and zero initial conditions, $\zeta^L(0)=\zeta^R(0)=0$, are prescribed. Thus, the proof is complete. 

\end{proof}

\begin{remark}
\textit{
For symmetric systems, the ODE system~\reff{potentialdropODE1} reduces to 
\begin{equation}
\displaystyle -C (\zeta^L) \frac{d \zeta^L}{dt}=2z_+^3 (V_+ + \zeta^L),
\end{equation}
where $C(u) =\sqrt{2}z_+^{3/2}\cosh( z_+u/2)$.
For asymmetric systems, one needs to numerically solve the ODE system~\eqref{potentialdropODE} for $\zeta^L(t)$ and $\zeta^R(t)$. Then, the functions $\overline{j}_0(t)$ and $A(t)$, as well as the inner and outer solutions, can be obtained accordingly.
Detailed analysis of the leading order approximation for the case of $n=1$ has also been investigated in literature for electrochemical systems~\cite{BTA:PRE:04,ji2018asymptotic}. 
}
\end{remark}

We next extend such an MAE approach to the stack-electrode model that mimics the porous structure of supercapacitors with multiple electrodes, i.e., $n>1$.

\subsection{The $n=2$ case}
\label{sec:n2}
For the case of $n=2$, one places another two electrodes at $x = \pm \tilde{\calL}$, where $ \tilde{\calL} =  \pm \calL/D$ and again the tilde is omitted for simplicity.  
One aims to investigate the leading order dynamics within the left half domain $[-1, 0]$ first.

\textbf{\textit{Outer solution}}--
\label{sec:bulk:n2}
Similar to the notations of the $n=1$ case, bulk solutions in domains $1_b$, and $2_{b}$ are denoted by $\overline{c}_\pm^1, \overline{c}_\pm^{2}$, $\overline{\Phi}^{1}$, and $\overline{\Phi}^{2}$. Referring to the analysis in Sec.~\ref{sec:n1}, one can directly obtain the leading order solutions of the PNP equations in each bulk area as follows:
\begin{equation}
\label{eq:BulkSolution:n2:1}
\left\{
\begin{aligned}
&\overline{c}^1_{\pm,0}(x,t)=|z_\mp|, \\
&\overline{\Phi}^1_0(x,t)=\overline{j}^1_0(t)x+A^1(t), -\calL< x <\calL,
\end{aligned}
\right.
\end{equation}
and
\begin{equation}
\label{eq:BulkSolution:n2:2}
\left\{
\begin{aligned}
&\overline{c}^{2}_{\pm,0}(x,t)=|z_\mp|, \\
&\overline{\Phi}^{2}_0(x,t)=\overline{j}^{2}_0(t)x+A^{2}(t), -1< x <-\calL,
\end{aligned}
\right.
\end{equation}
with initial currents $\overline{j}^1_0(t=0) = (V_+ - V_-)/(2\calL)$, $\overline{j}^{2}_0(t=0) =0$, $A^1(t=0) = (V_++V_-)/2$, and $A^{2}(t=0) = V_-$.

\textbf{\textit{Inner solution}}--
For the inner solutions, one first focuses on the inner layer of $(1,L)$. Denote the solution by
$\widehat{c}^1_\pm(y,t)$ and $\widehat{\Phi}^1(y,t)$ with a scale transformation $y = (x+\calL)/\epsilon$. The leading order solutions can be obtained by using the flux matching conditions between inner layers and outer layers of the sub-domain $1$:
\begin{equation}
\label{eq:IISolution:n2:1}
\left\{
\begin{aligned}
&\widehat{c}^{1,L}_{\pm,0}=|z_\mp|\exp({-z_\pm \widehat{\varphi}^{1,L}_0}),  \\
&-\frac{\partial^2 \widehat{\varphi}^{1,L}_0}{\partial y^2}=z_+\widehat{c}^{1,L}_{+,0}+z_-\widehat{c}^{1,L}_{-,0},
\end{aligned}
\right.
\end{equation}
where $\widehat{\varphi}^{1,L}_0(y,t)$ represents the leading-order term of the potential drop $\widehat{\varphi}^{1,L}(y,t) := \widehat{\Phi}^{1,L}(y,t) - \overline{\Phi}^{1}(-\calL,t)$. The boundary conditions are
$\widehat{\varphi}^{1,L}_0(0,t)=\zeta^{1,L}(t)\triangleq \widehat{\Phi}^{1,L}_0(0,t)-\overline{\Phi}^{1}_0(-\calL,t)$  and $\widehat{\varphi}^{1,L}_0(\infty,t)=0$.
Transformation variables $y_2 =(x+1)/\epsilon \in (0, \infty)$ and $y_1 = (x+\calL)/\epsilon \in (-\infty, 0)$ are used in the inner layers of $(2,L)$, and $(2,R)$, respectively.
Resorting to the above analysis and using matchings between bulk areas and the corresponding inner layers, give the following leading order inner solutions 
\begin{equation}
\left\{
\begin{aligned}
&\widehat{c}^{2,L}_{\pm,0}=|z_\mp|\exp({-z_\pm \widehat{\varphi}^{2,L}_0}),  \\
&-\frac{\partial^2 \widehat{\varphi}^{2,L}_0}{\partial y_2^2}=z_+\widehat{c}^{2,L}_{+,0}+z_-\widehat{c}^{2,L}_{-,0},
\end{aligned}
\label{eq:Poisson:n2:3}
\right.
\end{equation}
and
\begin{equation}
\left\{
\begin{aligned}
&\widehat{c}^{2,R}_{\pm,0}=|z_\mp|\exp({-z_\pm \widehat{\varphi}^{2,R}_0}),  \\
&-\frac{\partial^2 \widehat{\varphi}^{2,R}_0}{\partial y_1^2}=z_+\widehat{c}^{2,R}_{+,0}+z_-\widehat{c}^{2,R}_{-,0},
\end{aligned}
\label{eq:Poisson:n2:2}
\right.
\end{equation}
where $\widehat{c}^{2,L}_{\pm,0}$ 
 and $\widehat{c}^{2,R}_{\pm,0}$ 
 represent the leading order ionic concentrations
 in $(2,L)$, and $(2,R)$, respectively, and $\widehat{\varphi}^{2,L}_0$ and $
\widehat{\varphi}^{2,R}_0$ represent the leading order approximations of potential drops $\widehat{\varphi}^{2,L}:= \widehat{\Phi}^{2,L}(y_2,t) - \overline{\Phi}^{2}(-1,t)$ and $\widehat{\varphi}^{2,R}:= \widehat{\Phi}^{2,R}(y_1,t) - \overline{\Phi}^{2}(-\calL,t)$ in $(2,L)$ and $(2,R)$, respectively. Define $\zeta^{2,L}(t) \triangleq \widehat{\varphi}^{2,L}_0(y_2 =0,t)$ and $\zeta^{2,R}(t) \triangleq \widehat{\varphi}^{2,R}_0(y_1 =0,t)$. The potential matching condition gives that $\widehat{\varphi}^{2,L}_0(\infty,t) =0$ and $
\widehat{\varphi}^{2,R}_0(-\infty,t) =0$, thus the above Poisson--Boltzmann (PB) equations \eqref{eq:Poisson:n2:3} are determined by $\zeta^{2,L}(t)$, $\zeta^{2,R}(t)$, and zeros boundary conditions. 
Leading order solutions in right sub-domains can be derived similarly with potential drops $\zeta^{1,R}$, $\zeta^{2',L}$, and $\zeta^{2',R}$ being boundary conditions. These zeta potentials can be solved through ODE system \eqref{eq:ODE:n2:final1} in Lemma \ref{lemma:2}, and all the MAE solutions can be calculated subsequently.

\begin{lemma} \label{lemma:2}
For the $n=2$ case, all the zeta potentials satisfy the following closed system of ODEs
\begin{align}
\begin{cases}
\displaystyle -C^{2,L}\frac{d\zeta^{2,L}}{dt} = \alpha \frac{\zeta^{2,L} -\zeta^{1,L}}{1-\calL},\\
\displaystyle -2C^{1,L}\frac{d\zeta^{1,L}}{dt} =\alpha\left( \frac{\zeta^{1,L} -\zeta^{1,R}-V_-+V_+}{2\calL} - \frac{\zeta^{2,L} -\zeta^{1,L}}{1-\calL} \right),\\
\displaystyle -2C^{1,R}\frac{d\zeta^{1,R}}{dt} = \alpha\left(\frac{\zeta^{1,R} -\zeta^{2',R}}{1-\calL} - \frac{\zeta^{1,L} -\zeta^{1,R}-V_-+V_+}{2\calL}\right),\\
\displaystyle -C^{2',R}\frac{d\zeta^{2',R}}{dt} = \alpha\frac{\zeta^{1,R} -\zeta^{2',R}}{1-\calL},
\end{cases}
\label{eq:ODE:n2:final1}
\end{align} 
with zero initial conditions.

\end{lemma}

\begin{proof}
Consider the charging process, where the total diffuse charges $\widehat{Q}_0^{k,S}$, $k= 1, 2, 2'$ and $S=L, R$, are given by
\begin{align}
\left\{
\begin{aligned}
&\widehat{Q}_0^{2,L}
=\sqrt{2}\sqrt{-z_-e^{-z_+\zeta^{2,L}}+z_+e^{-z_-\zeta^{2,L}}+z_--z_+},\\
&\widehat{Q}_0^{1,L}=\sqrt{2}\sqrt{-z_-e^{-z_+\zeta^{1,L}}+z_+e^{-z_-\zeta^{1,L}}+z_--z_+},\\
&\widehat{Q}_0^{2,R} =\sqrt{2}\sqrt{-z_-e^{-z_+\zeta^{2,R}}+z_+e^{-z_-\zeta^{2,R}}+z_--z_+},\\
&\widehat{Q}_0^{1,R}= -\sqrt{2}\sqrt{-z_-e^{-z_+\zeta^{1,R}}+z_+e^{-z_-\zeta^{1,R}}+z_--z_+},\\
&\widehat{Q}_0^{2',L}= -\sqrt{2}\sqrt{-z_-e^{-z_+\zeta^{2',L}}+z_+e^{-z_-\zeta^{2',L}}+z_--z_+}, \\
&\widehat{Q}_0^{2',R}= -\sqrt{2}\sqrt{-z_-e^{-z_+\zeta^{2',R}}+z_+e^{-z_-\zeta^{2',R}}+z_--z_+}. 
\end{aligned}
\right.
\end{align}
Their time derivatives are given by
\begin{align}
\left\{
\begin{aligned}
&\frac{d \widehat{Q}^{1,L}_0}{dt}
=\alpha \overline{j}_0^1(t)- \lim_{y\rightarrow 0}\frac{1}{\epsilon}\left[\frac{\partial \widehat{\rho}_0^{1,L}}{\partial y}+\left(z_+^2\widehat{c}^{1,L}_{+,0}+z_{-}^2\widehat{c}^{1,L}_{-,0}\right) \frac{\partial \widehat{\Phi}_0^{1,L}}{ \partial y}\right],\\
&\frac{d \widehat{Q}^{2,L}_0}{dt} =\alpha\overline{j}^{2}_0(t),  \\
&\frac{d \widehat{Q}^{2,R}_0}{dt}
= \lim_{y_1 \rightarrow 0}\frac{1}{\epsilon}\left[\frac{\partial \widehat{\rho}_0^{2,R}}{\partial y_1}+\left(z_+^2\widehat{c}^{2,R}_{+,0}+z_{-}^2\widehat{c}^{2,R}_{-,0}\right) \frac{\partial \widehat{\Phi}_0^{2,R}}{ \partial y_1}\right] -\alpha\overline{j}^{2}_0(t), \\
&\frac{d \widehat{Q}^{1,R}_0}{dt}
=-\lim_{y\rightarrow 0}\frac{1}{\epsilon}\left[\frac{\partial \widehat{\rho}_0^{1,R}}{\partial y}+\left(z_+^2\widehat{c}^{1,R}_{+,0}+z_{-}^2\widehat{c}^{1,R}_{-,0}\right) \frac{\partial \widehat{\Phi}_0^{1,R}}{ \partial y}\right] -\alpha \overline{j}_0^1(t),\\
&\frac{d \widehat{Q}^{2',L}_0}{dt}
= \lim_{y_1 \rightarrow 0}\frac{1}{\epsilon}\left[\frac{\partial \widehat{\rho}_0^{2',L}}{\partial y_1}+\left(z_+^2\widehat{c}^{2',L}_{+,0}+z_{-}^2\widehat{c}^{2',L}_{-,0}\right) \frac{\partial \widehat{\Phi}_0^{2',L}}{ \partial y_1}\right]  + \alpha\overline{j}^{2'}_0(t),\\
&\frac{d \widehat{Q}^{2',R}_0}{dt} = -\alpha\overline{j}^{2'}_0(t),
\end{aligned}
\right.
\end{align}
where the \textit{matching of flux} between the bulk and inner layers has been used.

By the definition of the zeta potentials, {each current can be determined by two corresponding zeta potentials: $\overline{j}_0^1 = (\zeta^{1,L} -\zeta^{1,R}-V_-+V_+)/(2\calL)$. 
By the flux continuity at electrodes, one has}
\begin{align*}
&\lim_{y_1 \rightarrow 0}\frac{1}{\epsilon}\left[\frac{\partial \widehat{\rho}_0^{2,R}}{\partial y_1}+\left(z_+^2\widehat{c}^{2,R}_{+,0}+z_{-}^2\widehat{c}^{2,R}_{-,0}\right) \frac{\partial \widehat{\Phi}_0^{2,R}}{ \partial y_1}\right] = \lim_{y\rightarrow 0}\frac{1}{\epsilon}\left[\frac{\partial \widehat{\rho}_0^{1,L}}{\partial y}+\left(z_+^2\widehat{c}^{1,L}_{+,0}+z_{-}^2\widehat{c}^{1,L}_{-,0}\right) \frac{\partial \widehat{\Phi}_0^{1,L}}{ \partial y}\right],\\
&\lim_{y_1 \rightarrow 0}\frac{1}{\epsilon}\left[\frac{\partial \widehat{\rho}_0^{2',L}}{\partial y_1}+\left(z_+^2\widehat{c}^{2',L}_{+,0}+z_{-}^2\widehat{c}^{2',L}_{-,0}\right) \frac{\partial \widehat{\Phi}_0^{2',L}}{ \partial y_1}\right] = \lim_{y\rightarrow 0}\frac{1}{\epsilon}\left[\frac{\partial \widehat{\rho}_0^{1,R}}{\partial y}+\left(z_+^2\widehat{c}^{1,R}_{+,0}+z_{-}^2\widehat{c}^{1,R}_{-,0}\right) \frac{\partial \widehat{\Phi}_0^{1,R}}{ \partial y}\right].
\end{align*}
On the other hand, the continuous conditions of ionic concentrations at the internal electrodes give
\begin{align*}
\lim\limits_{y \rightarrow 0} \widehat{c}_{\pm,0}^{1,L}= \lim\limits_{y_1 \rightarrow 0} \widehat{c}_{\pm,0}^{2,R},~~
\lim\limits_{y \rightarrow 0} \widehat{c}_{\pm,0}^{1,R}= \lim\limits_{y_1 \rightarrow 0} \widehat{c}_{\pm,0}^{2'L},
\end{align*}
which indicates that
\begin{align*}
\zeta^{1,L} =\zeta^{2,R},~~
\zeta^{1,R} = \zeta^{2',L}, ~C^{1,L}=C^{2,R}, ~C^{1,R} = C^{2',L}.
\end{align*}
Therefore, one finally arrives at
\begin{align}
\begin{cases}
\displaystyle -C^{2,L}\frac{d\zeta^{2,L}}{dt} = \alpha \frac{\zeta^{2,L} -\zeta^{1,L}}{1-\calL},\\
\displaystyle -2C^{1,L}\frac{d\zeta^{1,L}}{dt} =\alpha\left( \frac{\zeta^{1,L} -\zeta^{1,R}-V_-+V_+}{2\calL} - \frac{\zeta^{2,L} -\zeta^{1,L}}{1-\calL} \right),\\
\displaystyle -2C^{1,R}\frac{d\zeta^{1,R}}{dt} = \alpha\left(\frac{\zeta^{1,R} -\zeta^{2',R}}{1-\calL} - \frac{\zeta^{1,L} -\zeta^{1,R}-V_-+V_+}{2\calL}\right),\\
\displaystyle -C^{2',R}\frac{d\zeta^{2',R}}{dt} = \alpha\frac{\zeta^{1,R} -\zeta^{2',R}}{1-\calL},
\end{cases}
\label{eq:ODE:n2:final}
\end{align} 
with zero initial conditions. Here, one 
has used the chain rule and the definition of the differential capacitance, that is, 
\begin{align*}
\frac{d \widehat{Q}^{k,S}_0}{dt} = -C^{k, S}\frac{d\zeta^{k,S}}{dt},
\end{align*}
with notations $C^{k, S}= C(\zeta^{k,S})$ for $k= 1, 2, 2'$ and $S=L, R$. Thus, the proof is finished.

\end{proof}

\subsection{General case: $n>2$}
\label{sec:AsymSolutionGeneral}

For the general case, one
denotes the leading-order outer solutions in each sub-domain as $\overline{c}_{\pm,0}^{k}$ and $\overline{\Phi}_0^{k}$, which can be analogously derived as
\begin{equation}
\label{eq:BulkSolution:nn:k}
\left\{
\begin{aligned}
&\overline{c}^k_{\pm,0}(x,t)=|z_\mp|, ~-x_{k}< x <-x_{k-1}, ~k=2, \cdots, n, \\
&\overline{\Phi}^k_0(x,t)=\overline{j}^k_0(t)x+A^k(t), ~ -x_{k}< x <-x_{k-1},~ k=2, \cdots, n,
\end{aligned}
\right.
\end{equation}
with initial values $\overline{j}^k_0(t=0) = 0$ and $A^k(t=0) = V_-$. On the other hand, the leading-order solutions in the right sub-domains are given by
\begin{equation}
\label{eq:BulkSolution:nn:kprime}
\left\{
\begin{aligned}
&\overline{c}^{k'}_{\pm,0}(x,t)=|z_\mp|, ~-x_{k'-1}< x < x_{k'}, ~k=2, \cdots, n, \\
&\overline{\Phi}^{k'}_0(x,t)=\overline{j}^{k'}_0(t)x+A^{k'}(t), ~x_{k'-1}< x < x_{k'}, ~k=2, \cdots, n,
\end{aligned}
\right.
\end{equation}
with initial values $\overline{j}^{k'}_0(t=0) = 0$ and $A^{k'}(t=0) = V_+$. For the solutions in the bulk area $1$ (or denoted by $1'$), the leading-order outer solution is given by 
\begin{equation}
\label{eq:BulkSolution:nn:1}
\left\{
\begin{aligned}
&\overline{c}^1_{\pm,0}(x,t)=|z_\mp|,~  -\calL< x < \calL, \\
&\overline{\Phi}^1_0(x,t)=\overline{j}^1_0(t)x+A^1(t), ~ -\calL< x < \calL,
\end{aligned}
\right.
\end{equation}
with initial values $\overline{j}^1_0(t=0) = (V_+ - V_-)/(2\calL)$ and $A^1(t=0) = (V_+ + V_-)/2$. 
Obviously, each potential function is linear in the bulk regions, and the ionic concentrations are constant.

 For the left inner solutions in the left sub-domains $(k,L), k=1, \cdots, n$, define $y_{2k-1}=(x+x_k)/\epsilon \in (0, +\infty)$, and use $\widehat{c}^{k,L}_\pm(y_{2k-1},t)$, $\widehat{\Phi}^{k,L}(y_{2k-1},t)$, and $\widehat{\varphi}^{k,L}(y_{2k-1},t) := \widehat{\Phi}^{k,L}(y_{2k-1},t) - \overline{\Phi}^{k}(-x_k,t)$ to represent the ionic concentrations, potential distributions, and potential drops, respectively. The leading-order solutions satisfy
\begin{equation}
\label{eq:IISolution:nn:kL}
\left\{
\begin{aligned}
&\widehat{c}^{k,L}_{\pm,0}=|z_\mp|\exp({-z_\pm \widehat{\varphi}^{k,L}_0}),  \\
&-\frac{\partial^2 \widehat{\varphi}^{k,L}_0}{\partial y_{2k-1}^2}=z_+\widehat{c}^{k,L}_{+,0}+z_-\widehat{c}^{k,L}_{-,0},
\end{aligned}
\right.
\end{equation} 
with boundary conditions $\widehat{\varphi}^{k,L}_0(y_{2k-1}=0,t) = \zeta^{k,L}$ and $\widehat{\varphi}^{k,L}_0(y_{2k-1}=\infty,t)=0$, $k=1, \cdots, n$. Here $\zeta^{k,L}$ is the zeta potential in the inner layer of the left sub-domain $(k,L)$.  Similarly, with the scale transformation $y_{2(k-1)}=(x+x_{k-1})/\epsilon \in (-\infty, 0),$ {one obtains that} the potential drop in right inner region of $(k,R)$, i.e., $\widehat{\varphi}^{k,R}_0(y_{2(k-1)},t) := \widehat{\Phi}^{k,R}_0(y_{2(k-1)},t) - \overline{\Phi}^{k}_0(-x_{k-1},t)$, which satisfies
\begin{equation}
\label{eq:IISolution:nn:kR}
\left\{
\begin{aligned}
&\widehat{c}^{k,R}_{\pm,0}=|z_\mp|\exp({-z_\pm \widehat{\varphi}^{k,R}_0}),  \\
&-\frac{\partial^2 \widehat{\varphi}^{k,R}_0}{\partial y_{2(k-1)}^2}=z_+\widehat{c}^{k,R}_{+,0}+z_-\widehat{c}^{k,R}_{-,0},
\end{aligned}
\right.
\end{equation}
with boundary conditions $\widehat{\varphi}^{k,R}_0(y_{2(k-1)}=0,t) = \zeta^{k,R}$ and $\widehat{\varphi}^{k,R}_0(y_{2(k-1)}=-\infty,t)=0$, $k=2, \cdots, n$. Here $\zeta^{k,R}$ is the zeta potential in the inner layer of the right sub-domain $(k,R)$.
The leading-order inner solutions in the right sub-domains can also be analogously derived. We use the {same notations as those} in the left sub-domain but with a prime on the superscript. Define the transformation variable $y_{2(k'-1)}=({x_{k'-1}-x})/{\epsilon}$, $k'=2, \cdots, n$, and the potential drop $\widehat{\varphi}^{k',L}(y_{2(k'-1)},t) := \widehat{\Phi}^{k',L}(y_{2(k'-1)},t) - \overline{\Phi}^{k'}(x_{k'-1},t)$. Then one can analogously derive
\begin{equation}
\label{eq:IISolution:nn:kprimeL}
\left\{
\begin{aligned}
&\widehat{c}^{k',L}_{\pm,0} =|z_\mp|\exp({-z_\pm \widehat{\varphi}^{k',L}_0}),  \\
&-\frac{\partial^2 \widehat{\varphi}^{k',L}_0}{\partial y_{2(k'-1)}^2} =z_+\widehat{c}^{k',L}_{+,0}+z_-\widehat{c}^{k',L}_{-,0}, 
\end{aligned}
\right.
\end{equation}  
with boundary conditions $\widehat{\varphi}^{k',L}_0(y_{2(k'-1)}=0,t) =\zeta^{k',L} $ and $\widehat{\varphi}^{k',L}_0(y_{2(k'-1)}=-\infty,t)=0$.  For the right inner solutions in $(k',R)$, introduce the scale transformation $y_{2k'-1}=(x_{k'}-x)/\epsilon$ and the potential drop $\widehat{\varphi}^{k',R}(y_{2k'-1},t) := \widehat{\Phi}^{k',R}(y_{2k'-1},t) - \overline{\Phi}^{k'}(x_{k'},t)$. Then one can analogously derive
\begin{equation}
\label{eq:IISolution:nn:kprimeR}
\left\{
\begin{aligned}
&\widehat{c}^{k',R}_{\pm,0}=|z_\mp|\exp({-z_\pm \widehat{\varphi}^{k',R}_0}),  \\
&-\frac{\partial^2 \widehat{\varphi}^{k',R}_0}{\partial y_{2k'-1}^2}=z_+\widehat{c}^{k',R}_{+,0}+z_-\widehat{c}^{k',R}_{-,0}, 
\end{aligned}
\right.
\end{equation} 
with boundary conditions $\widehat{\varphi}^{k',R}_0(y_{2k'-1}=0,t) = \zeta^{k',R}$ and $\widehat{\varphi}^{k',R}_0(y_{2k'-1}=\infty,t)=0$, $k=1, \cdots, n$.

All these leading-order solutions are determined by  zeta potentials $\zeta^{k',L}, \zeta^{k',R}, \zeta^{k,R}, \zeta^{k,L}$, which can be solved through the ODE system \eqref{eq:matrix:form} in Theorem \ref{theorem:n}.

\begin{theorem}\label{theorem:n}
For the system with $2n$ electrodes, all the zeta potentials satisfy the following closed system of ODEs
\begin{equation}
\bm{C}(\zeta)
\dot{\zeta} =\alpha(-\mathcal{T}\zeta + Y),
\label{eq:matrix:form}
\end{equation}
with zero initial conditions. Here $\zeta =[\zeta^{n,L}, \cdots, \zeta^{1,L}, \zeta^{1',R}, \cdots, \zeta^{n',R}]^T$,\\
 ${\bf C}(\zeta)=\text{diag}\left\{C (\zeta^{n,L}), 2C(\zeta^{n-1,L}), \cdots, 2C(\zeta^{1,L}), 2C(\zeta^{1',R}), \cdots, 2C(\zeta^{(n-1)',R}, C(\zeta^{n',R})\right\}$,
{\footnotesize{
\begin{equation*}
\mathcal{ T}=
\begin{pmatrix}
\frac{1}{h} & -\frac{1}{h}& & &\\
-\frac{1}{h} & \frac{2}{h}& -\frac{1}{h}& & \\
&  \ddots & \ddots & \ddots & \\
& &-\frac{1}{h} & (\frac{1}{h}+\frac{1}{2\calL})& -\frac{1}{2\calL}&\\
& & & -\frac{1}{2\calL}&(\frac{1}{h}+\frac{1}{2\calL}) &-\frac{1}{h} & \\
& & & & -\frac{1}{h} &\frac{2}{h}&-\frac{1}{h} & &\\
& & & & & \ddots& \ddots & \ddots  & &\\
& & & & & & -\frac{1}{h} &\frac{1}{h}
\end{pmatrix}_{2n \times 2n},~~
Y=
\begin{pmatrix}
0\\
0\\
\vdots\\
-\frac{V_+-V_-}{2\calL}\\
-\frac{V_--V_+}{2\calL}\\
0\\
\vdots\\
0\\
0
\end{pmatrix}_{2n \times 1}.
\end{equation*}
}}

\end{theorem}

\begin{proof}

Consider the leading-order total diffuse charge in each inner layer:
 \begin{align}
 \left\{
 \begin{aligned}
&\displaystyle \widehat{Q}_0^{k,L}= \sqrt{2}\sqrt{\displaystyle -z_-e^{-z_+\zeta^{k,L}}+z_+e^{-z_-\zeta^{k,L}}+z_--z_+}, ~k=1, \cdots, n,  \\
&\displaystyle \widehat{Q}_0^{k,R} = \sqrt{2}\sqrt{\displaystyle -z_-e^{-z_+\zeta^{k,R}}+z_+e^{-z_-\zeta^{k,R}}+z_--z_+}, ~k=2, \cdots, n,  \\
&\displaystyle \widehat{Q}_0^{k',R}= -\sqrt{2}\sqrt{\displaystyle -z_-e^{-z_+\zeta^{k',R}}+z_+e^{-z_-\zeta^{k',R}}+z_--z_+}, ~k=1, \cdots, n, \\
&\displaystyle \widehat{Q}_0^{k',L} = -\sqrt{2}\sqrt{\displaystyle -z_-e^{-z_+\zeta^{k',L}}+z_+e^{-z_-\zeta^{k',L}}+z_--z_+}, ~k=2, \cdots, n. 
\end{aligned}
\right.
 \end{align}
 For the total diffuse charge in the $(n,L)$ inner layer, one uses zero flux boundary condition and the flux matching between inner and outer solutions to get the following equation
\begin{align}
\frac{d \widehat{Q}^{n,L}_0}{dt} &=\int_0^{\infty} \left( z_{+}\frac{\partial\widehat{ c}^{n,L}_{+,0}}{\partial t}+z_{-}\frac{\partial\widehat{c}^{n,L}_{-,0}}{\partial t} \right) dy_{2n-1} = \alpha\overline{j}^{n}_0(t).  \label{eq:diff:nn:nL}
\end{align}
For the inner layers in $(k,L)$, $k=1, \cdots, n-1$, one has
\begin{align}
\frac{d \widehat{Q}^{k,L}_0}{dt} &=\int_0^{\infty} \left( z_{+}\frac{\partial\widehat{ c}^{k,L}_{+,0}}{\partial t}+z_{-}\frac{\partial\widehat{c}^{k,L}_{-,0}}{\partial t} \right) dy_{2k-1} = \alpha\overline{j}^{k}_0(t) -\lim_{y_{2k-1} \rightarrow 0}\frac{1}{\epsilon} F^{k,L},  \label{eq:diff:nn:kL}
\end{align}
where the flux $F^{k,L} = \frac{\partial \widehat{\rho}_0^{k,L}}{\partial y_{2k-1}}+\left(z_+^2\widehat{c}^{k,L}_{+,0}+z_{-}^2\widehat{c}^{k,L}_{-,0}\right) \frac{\partial \widehat{\Phi}_0^{k,L}}{ \partial y_{2k-1}}$. For inner layers in $(k,R)$, $k=2, \cdots, n$, one can derive
\begin{align}
\frac{d \widehat{Q}^{k,R}_0}{dt} &=\int^0_{-\infty} \left( z_{+}\frac{\partial\widehat{ c}^{k,R}_{+,0}}{\partial t}+z_{-}\frac{\partial\widehat{c}^{k,R}_{-,0}}{\partial t} \right) dy_{2(k-1)} = -\alpha\overline{j}^{k}_0(t) +\lim_{y_{2(k-1)} \rightarrow 0}\frac{1}{\epsilon}F^{k,R},  \label{eq:diff:nn:kR}
\end{align}
where the flux $F^{k,R}$ is analogously defined. The flux continuity gives
$$\lim_{y_{2(k-1)} \rightarrow 0} F^{k,R} = \lim_{y_{2k-1} \rightarrow 0} F^{k,L},~ k=1, \cdots, n-1.$$
The ion concentration continuity at each internal electrode leads to
 $$\zeta^{k,L} =\zeta^{k+1,R}, ~k=1,\cdots, n.$$ 
 Adding Eqs.~\eqref{eq:diff:nn:kL} and\eqref{eq:diff:nn:kR} together results in
\begin{align}
\frac{d \widehat{Q}^{k,L}_0}{dt} +\frac{d \widehat{Q}^{k+1,R}_0}{dt} = \alpha\left[ \overline{j}^{k}_0(t)-\overline{j}^{k+1}_0(t)\right].
\end{align}
Therefore, there are $n+1$ unknown potential drops: $\zeta^{1,R}$ and  $\zeta^{k,L}$ for $k=1, \cdots, n$. It should be noted that the currents can be represented by corresponding zeta potentials:
\begin{align}
\left\{
\begin{aligned}
\overline{j}^{1}_0(t)&=\frac{\zeta^{1,L} - \zeta^{1,R}-V_-+V_+}{2\calL},\\
\overline{j}^{k}_0(t)&=\frac{\zeta^{k,L} -\zeta^{k,R}}{h}=\frac{\zeta^{k,L} -\zeta^{k-1,L}}{h},  k=2, \cdots, n.
\end{aligned}
\right.
\end{align}

For the right sub-domains, one denotes the flux as $F^{k',L}$ and $F^{k',R}$, and uses the flux continuity and ion concentration continuity at each internal electrode in the sub-domain to derive following equations
\begin{align}
\left\{
\begin{aligned}
&\frac{d \widehat{Q}^{n',R}_0}{dt} =-\alpha\overline{j}^{n'}_0(t), \\
&\frac{d \widehat{Q}^{k',R}_0}{dt} +\frac{d \widehat{Q}^{k'+1,L}_0}{dt} = \alpha\left[ -\overline{j}^{k'}_0(t)+\overline{j}^{k'+1}_0(t)\right],~ k' =1, \cdots, n-1,
\end{aligned}
\right.
\end{align}
where the currents can be represented as
\begin{align}
\left\{
\begin{aligned}
\overline{j}^{1'}_0(t)&=\frac{\zeta^{1,L} - \zeta^{1,R}-V_-+V_+}{2\calL},\\
\overline{j}^{k'}_0(t)&=\frac{\zeta^{k',L} -\zeta^{k',R}}{h}=\frac{\zeta^{k'-1,R} -\zeta^{k',R}}{h}, k'=2, \cdots, n.
\end{aligned}
\right.
\end{align}
Therefore, there are another $n-1$ unknown zeta potentials $\zeta^{k',R}$ for $k=2, \cdots, n$.

Combining the right and left half systems together, {one obtains the following ODE system that involves $2n$ unknown zeta potentials:}
\begin{align}
\begin{cases}
& \displaystyle -C^{n,L}\frac{d\zeta^{n,L}}{dt} =\displaystyle \alpha \frac{\zeta^{n,L} -\zeta^{n-1,L}}{h},\\
&\displaystyle -2C^{k,L}\frac{d\zeta^{k,L}}{dt} = -\alpha\frac{\zeta^{k-1,L} -2\zeta^{k,L}+\zeta^{k+1,L}}{h} ,~ k=2, \cdots, n-1,\\
&\displaystyle -2C^{1,L}\frac{d\zeta^{1,L}}{dt}=\alpha\left(\frac{\zeta^{1,L} -\zeta^{1',R}-V_-+V_+}{2\calL} - \frac{\zeta^{2,L} -\zeta^{1,L}}{h} \right),\\
&\displaystyle -2C^{1',R}\frac{d\zeta^{1',R}}{dt} = \alpha\left(\frac{\zeta^{1',R} -\zeta^{2',R}}{h} + \frac{-\zeta^{1,L} +\zeta^{1',R}+V_--V_+}{2\calL}\right),\\
&\displaystyle -2C^{k',R}\frac{d\zeta^{k',R}}{dt}= \alpha
\frac{-\zeta^{k'-1,R} +2\zeta^{k',R}-\zeta^{k'+1,R}}{h}, ~ k'=2, \cdots, n-1,\\
& \displaystyle -C^{n',R}\frac{d\zeta^{n',R}}{dt} = - \alpha\frac{\zeta^{n'-1,R} -\zeta^{n',R}}{h},
\end{cases}
\end{align}
with zero initial conditions for zeta potentials.
Alternatively, one can display these equations in a matrix form, which completes the proof.

\end{proof}

\begin{remark}
\textit{
In particular,  we remark that the derived ODE system~\reff{eq:matrix:form} has a similar form as an equivalent circuit model developed in the work~\cite{lian2020blessing}, if the same dimensionless parameters are used. In contrast to the coefficients that are treated as fitting parameters in the model, the coefficients in our model, such as the differential capacitances and ionic valences, are derived from the asymptotic analysis of the PNP equations with physical interpretations. In addition, the derived ODE system~\reff{eq:matrix:form} is able to describe asymmetric electrolytes, and takes nonlinear response into account by nonlinear potential-dependent differential capacitances. While the ionic and potential distributions next to electrodes are not available in the equivalent circuit model, our generalized model demonstrates that they can be found by further solving the PB equations.
}
\end{remark}

\begin{remark} 
\textit{
 {The MAE method is also applicable to modified PNP equations with steric and inhomogeneous dielectric effects. Modified PB equations with steric and inhomogeneous dielectric effects can be derived in replacement of \reff{eq:IISolution:nn:kprimeL} and \reff{eq:IISolution:nn:kprimeR}. }
 }
 \end{remark}

\section{Timescale analysis}\label{sec:generalizedRCtime}

\subsection{Leading-order approximation}
The leading-order inner and outer solutions are connected by the ODE system~\reff{eq:matrix:form}, which dictates the charging timescales of the parallel stack-electrode model. Following the terminology used in the work~\cite{lian2020blessing}, we call the charging timescale as the generalized RC timescale $\tau_{n}$, which is found by solving an eigenvalue problem. We first have the following Lemma on the eigenvalues of ${C(\zeta^\infty)}^{-1}\mathcal{T}$, with $\zeta^\infty$ being the equilibrium solution of the system \eqref{eq:matrix:form}.
\begin{lemma}\label{lemma:SimpleEigs}
The matrix ${C(\zeta^\infty)}^{-1}\mathcal{T}$ has one zero eigenvalue and $2n-1$ distinct positive eigenvalues.
\end{lemma}
\begin{proof}
It is easy to know that $0$ is an eigenvalue to the matrix ${C(\zeta^\infty)}^{-1}\mathcal{T}$. By the Gershgorin circle theorem, one finds that all eigenvalues of the matrix are non-negative.  Also, it is readily to show that the matrix is similar to a symmetric tridiagonal matrix whose eigenvalues are distinct. Thus, the proof is complete.
\end{proof}
\begin{theorem}\label{thm:generalized RCtime}
The timescale of the charging process governed by the ODE system \reff{eq:matrix:form} is given by 
\begin{align}
\tau_n =\frac{1}{\lambda_c},
\end{align}
where $\lambda_c$ is the smallest positive eigenvalue of the matrix $\alpha {C(\zeta^\infty)}^{-1}\mathcal{T}$.

\end{theorem}

\begin{proof}
The equilibrium solution $\zeta^{\infty}$ of the system~\reff{eq:matrix:form} is uniquely determined by the initial conditions and $\mathcal{T}\zeta^{\infty} =Y$. It follows from the Lemma~\ref{lemma:SimpleEigs} that ${C(\zeta^\infty)}^{-1}\mathcal{T}$ has one zero eigenvalue and $2n-1$ distinct positive eigenvalues. Infinitesimal perturbation of the equilibrium state satisfies a linearized system of the ODE system~\reff{eq:matrix:form}, and would not damp in the direction of the eigenvector of ${C(\zeta^\infty)}^{-1}\mathcal{T}$ corresponding to the zero eigenvalue. To study the relaxation timescale around the equilibrium state, we first project the ODE system~\reff{eq:matrix:form} to a subspace generated by the rest of eigenvectors. The relaxation timescale, described by the projected ODE system, can be used to estimate the charging timescale of the original ODE system.   

Since ${C(\zeta^\infty)}^{-1}\mathcal{T}$ has distinct eigenvalues, one has the diagonalization
\[
 \bm{Q}^{-1} {C(\zeta^\infty)}^{-1} \mathcal{T}  \bm{Q}= \bm{D} ,
\]
where $\bm{Q}=(v_1, v_2, \cdots, v_{2n})$ is an invertible matrix with $v_k$, $k=1, \cdots, 2n$, being the eigenvectors of ${C(\zeta^\infty)}^{-1}\mathcal{T}$, and $\bm{D}=\text{diag}(\lambda_1, \lambda_2, \cdots, \lambda_{2n})$ with the $\lambda_1=0$ and $\lambda_k>0$, $k=2, \cdots, 2n$, being the corresponding eigenvalues. Define a projection matrix $\bm{P}=\bm{S} \bm{Q}^{-1}$, with 
\[
\bm{S}=\begin{pmatrix}
0 & 1 & & & &  \\
   & 0 & 1& & &   \\
   &    & \ddots & \ddots &   \\
   &    &            &0 &1 & \\
   &    &            &   &0 &1
\end{pmatrix}_{(2n-1) \times 2n}.
\] 
Applying the projection matrix, one has 
\begin{equation}\label{Projected}
\bm{P} \dot{\zeta} =\alpha \left[-\bm{P} \bm{C}(\zeta)^{-1}\mathcal{T}\zeta + \bm{P} \bm{C}(\zeta)^{-1}Y\right].
\end{equation}
By the Hartman-Grobman theorem, the dynamics of a nonlinear system near a hyperbolic equilibrium point can be studied through linearized equations. Consider an infinitesimal perturbation of the equilibrium state: $\zeta^{\delta} (t)=\zeta^{\infty} + \delta q(t)$, where $\delta$ is a small number. Substituting $\zeta^{\delta}$ into \reff{Projected} and using the Taylor expansion, one obtains a linearized system around $\zeta^{\infty}$ by ignoring higher order terms:
\[
\bm{P} \dot{q} = -\alpha  \bm{P} \bm{C}(\zeta^\infty)^{-1} \mathcal{T} q.
\]
By the fact that $\bm{D}= \bm{D} \bm{S}^T \bm{S}$, one has
\[
\begin{aligned}
\bm{P} \dot{q}=  -\alpha  \bm{P}  \bm{Q} \bm{D}  \bm{S}^T \bm{P}  q.\\
\end{aligned}
\]
Let $\mu(t)= \bm{P}  q(t)$. Thus, the damping of the projected perturbation $\mu$ is governed by 
\begin{equation}\label{EigProb}
\begin{aligned}
\dot{\mu}= -\alpha \bm{P} \bm{Q}  \bm{D} \bm{S}^T \mu =-(\alpha  \bm{S} \bm{D} \bm{S}^T) \mu :=-\bm{M} \mu. 
\end{aligned}
\end{equation}
Thus the timescale of the charging process $\tau_n$ is given by $\tau_n =\frac{1}{\lambda_c}$, where $\lambda_c$ is the smallest eigenvalue of $\bm{M}$.
By the definition of the matrix $\bm{S}$, one finds that $\lambda_c$ is the smallest positive eigenvalue of the matrix $\alpha {C(\zeta^\infty)}^{-1} \mathcal{T}$. This completes the proof.
\end{proof}

\begin{remark}
\textit{
For a symmetric system, one can find that $\bm{C}(\zeta^\infty)$ is a scalar matrix. 
As a special case, when $h=2\calL$, i.e., {all stacks are equally spaced in the whole device}, we find that the smallest eigenvalue of $\bm{M}$ is given by
\[
\lambda_c= \frac{\alpha}{C(V_+)} \left[ 2 - 2 \cos(\frac{\pi}{2n}) \right] \sim \frac{\pi^2\alpha}{4n^2C(V_+)}=\frac{\pi^2 z_+^3}{2n^2C(V_+)},~ \text{as}~ n \to +\infty.
\]
Thus, the timescale of charging $\tau_{n}$ in such a scenario is equal to $\left(2n^2C(V_+)\right)/(\pi^2 z_+^3) \tau_c$. Comparing with the two-plate charging timescale $C(V_+)/(2z_+^3) \tau_c$, one finds that the charging process slows down significantly as the number of stack electrodes increases.}
\end{remark}
{The $\mathcal{T}$ matrix in the ODE system \reff{eq:matrix:form} resembles the difference matrix discretized from the Laplacian operator in 1D~\cite{janssen2021transmission}. Such resemblance helps to establish the connection between the stack-electrode model and transmission line (TL) type of models~\cite{de1963porous, GuptaZukStone_PRL2020, henrique2022charging, henrique2022impact}.  A linear timescale scaling on the number of stacks is predicted for large $n$~\cite{janssen2021transmission, lian2020blessing}.}

\subsection{Diffusion relaxation timescale} 
\label{sec:diffusiontime}
In addition to the timescale found in the leading-order approximation, another timescale emerges in numerical simulations of the PNP equations with large applied voltages~\cite{lian2020blessing}. Here, we use technique developed in~\cite{BTA:PRE:04} to discuss the existence of a second charging timescale, i.e., the diffusion relaxation timescale.

\textbf{The $n=1$ case.} Consider the excess salt of inner layer compared to the bulk area for a symmetric salt with $z_+ =-z_- =z$. Introduce one variable $c =z_+c_+ - z_-c-$, which represents the local salt concentration. One can re-formulate the original PNP equations with $c$ and $\rho$.
For the inner solutions, we use the scale transformation $y = (x+1)/\epsilon$ and reformulate the NP equations as follows: 
\begin{align}
\left\{
\begin{aligned}
&\epsilon \frac{\partial \widehat{c}}{\partial t} = \frac{\partial}{\partial y}\left(\frac{\partial \widehat{c}}{\partial y} +z\widehat{\rho}\frac{\partial \widehat{\Phi}}{\partial y}\right),\\
&\epsilon \frac{\partial \widehat{\rho}}{\partial t} = \frac{\partial}{\partial y}\left(\frac{\partial \widehat{\rho}}{\partial y}+z\widehat{c}\frac{\partial \widehat{\Phi}}{\partial y}\right),\\
\end{aligned}
\right.
\end{align}
Consider the total excess concentration $
\widehat{w}(t) = \int_0^\infty [\widehat{c}(y,t) - \overline{c}_0(-x_1,t)]dy
$. It follows from the matching of fluxes~\cite{BTA:PRE:04} that
\begin{align*}
\frac{d\widehat{w}}{dt} = &\int_0^\infty \frac{\partial \widehat{c}}{\partial t} dy = \lim_{x\rightarrow -1} \left( \frac{\partial \overline{c}}{\partial x}+ z\overline{\rho} \frac{\partial \overline{\Phi}}{ \partial x} \right).
\end{align*}
Following the analysis in Sect.~\ref{sec:AsympSolutions}, one has
\begin{align*}
\frac{d\widehat{w}_0}{dt} = \lim_{x\rightarrow -1} \epsilon \frac{\partial \overline{c}_1}{\partial x},
\end{align*}
where one has used the fact that the bulk charge density is zero up to order $O(\epsilon^2)$. 
This indicates that there exists a new, longer timescale $\overline{\tau}=\tau_c/\epsilon = D^2/\mathcal{D}_0$~\cite{BTA:PRE:04}.

\textbf{The $n=2$ case.} Similar to the $n=1$ case, we consider the total excess concentration in each inner layer for a symmetric system and use flux matching condition to derive
\begin{align}
\left\{
\begin{aligned}
&2 \frac{d\widehat{w}_0^{1,L}}{dt}=\lim_{x\rightarrow -\calL} \left( \epsilon \frac{ \partial\overline{c}_1^{1}}{\partial x}- \epsilon \frac{ \partial\overline{c}_1^{2}}{\partial x} \right),\\
&\frac{d\widehat{w}_0^{2,L}}{dt} = \lim_{x\rightarrow -1} \epsilon \frac{ \partial\overline{c}_1^{2}}{\partial x},
\end{aligned}
\right.
\end{align}
where one has used the fact that
$\widehat{w}_0^{1,L} = \widehat{w}_0^{2,R}$, due to the continuity of ion concentrations across the electrode. 
Analytical solution to this ODE system is not available~\cite{BTA:PRE:04}. However, similar to the analysis for the $n=1$ case, this ODE system indicates the existence of the new timescale $\overline{\tau}$.

For general cases with $n>2$, we can extend the analysis for $n=2$ to obtain a large ODE system with $n$ unknowns, which is omitted here for simplicity. The large ODE system again indicates the emergence of the second timescale $\overline{\tau}$. Note that this is the diffusion timescale and is independent of the number of electrode stacks $n$. The $n$-independence of such diffusion timescale has also been confirmed in Refs.~\cite{lian2020blessing,janssen2021transmission}, by fitting against numerical simulations of the PNP equations.

\section{Numerical results}
\label{sec:numerical:results}

Our asymptotic analysis for the parallel stack-electrode model is validated by comparing with the finite-difference solution in asymmetric and symmetric systems. The charging timescale is analyzed by numerically solving the eigenvalue problem~\reff{EigProb} obtained by the leading-order approximation. Finally, simulations are performed to verify the existence of a longer timescale $\overline{\tau}$ revealed by our asymptotic analysis in Section \ref{sec:diffusiontime}.

\subsection{Accuracy and dynamics}
We evaluate the accuracy of the leading-order approximation of the MAE against direct numerical solutions of PNP equations via a finite difference method~\cite{LJX:SIAP:2018}, which is second-order accurate in space and time.  Consider charging dynamics of a symmetric salt with $z_+ = -z_- =1$ under a low applied potential, $V_\pm = \pm 0.2$, with parameters $n=5$, $H = \calL =0.5$, $h=0.125$, and $\epsilon=0.005$, for which the MAE method is valid. In order to investigate the full dynamics of the parallel electrode-stack model, the time-dependent matching approach in Section~\ref{sec:AsympSolutions} is used to obtain asymptotic solution up to $T=60$, at which the system has reached equilibrium. The relative errors between the asymptotic solution and finite-difference solution of ionic concentrations and the potential distribution are shown in Fig.~\ref{fig:symmetric:com} (a). The relative errors show that the leading-order approximation of the MAE has quite good accuracy. Snapshots of the electrostatic potential distribution at different times in Fig.~\ref{fig:symmetric:com} (b) demonstrate that the asymptotic solution can accurately capture the whole charging dynamics, even in the sharp transitions in boundary layers.

\begin{figure}[H]
  \centering
\includegraphics[width=0.4\textwidth]{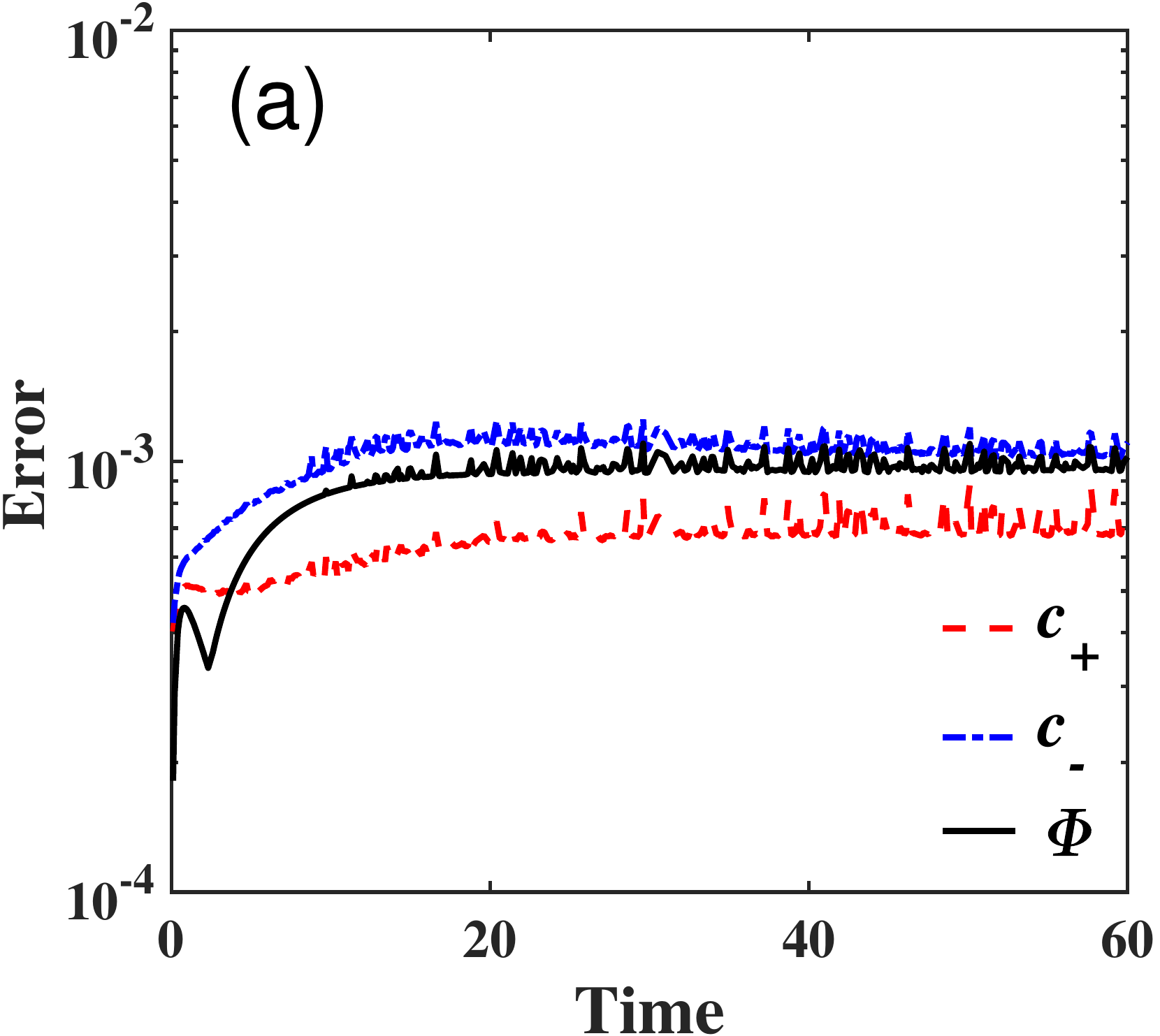}
  \includegraphics[width=0.395\textwidth]{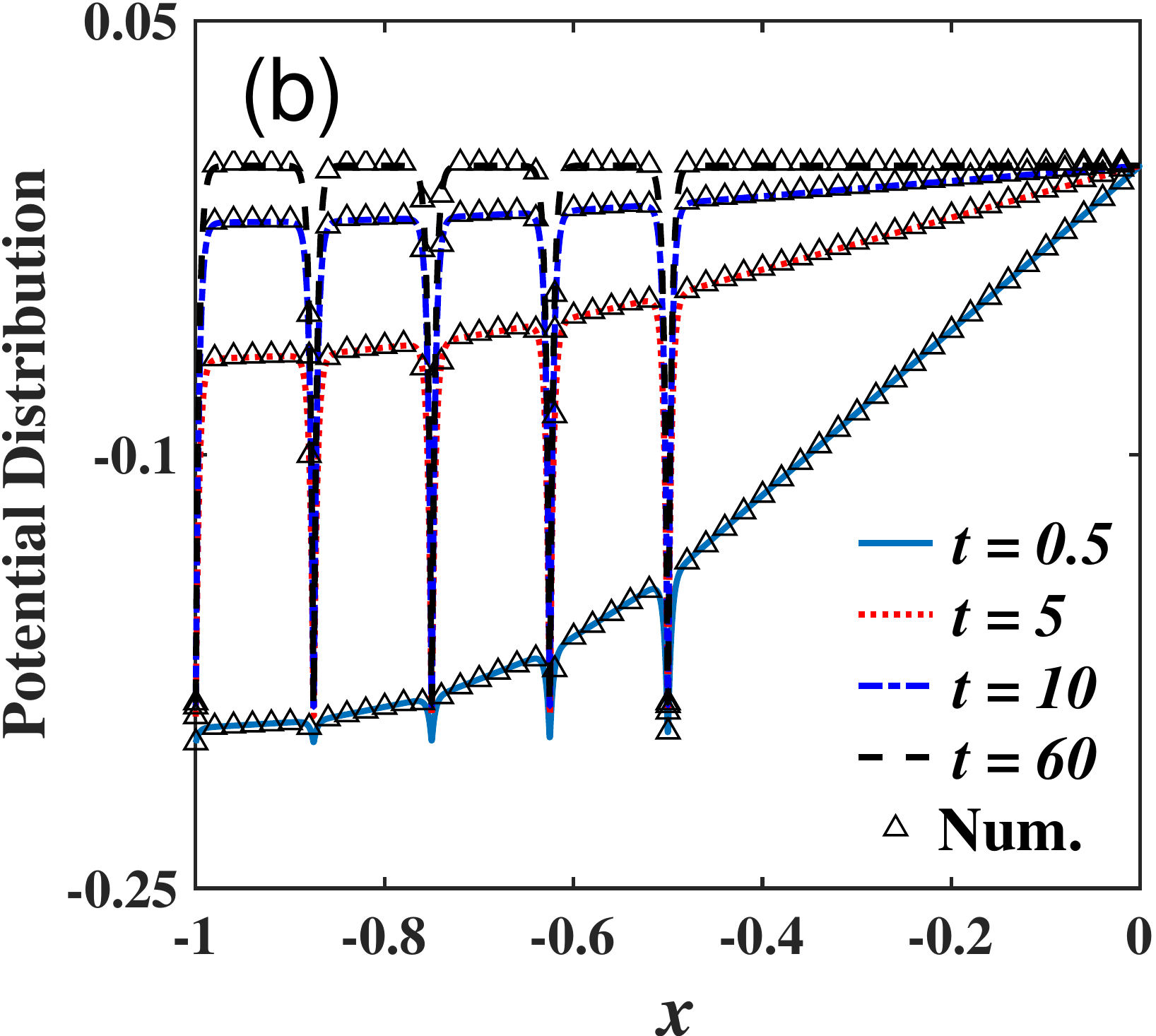}\\
   \includegraphics[width=0.4\textwidth]{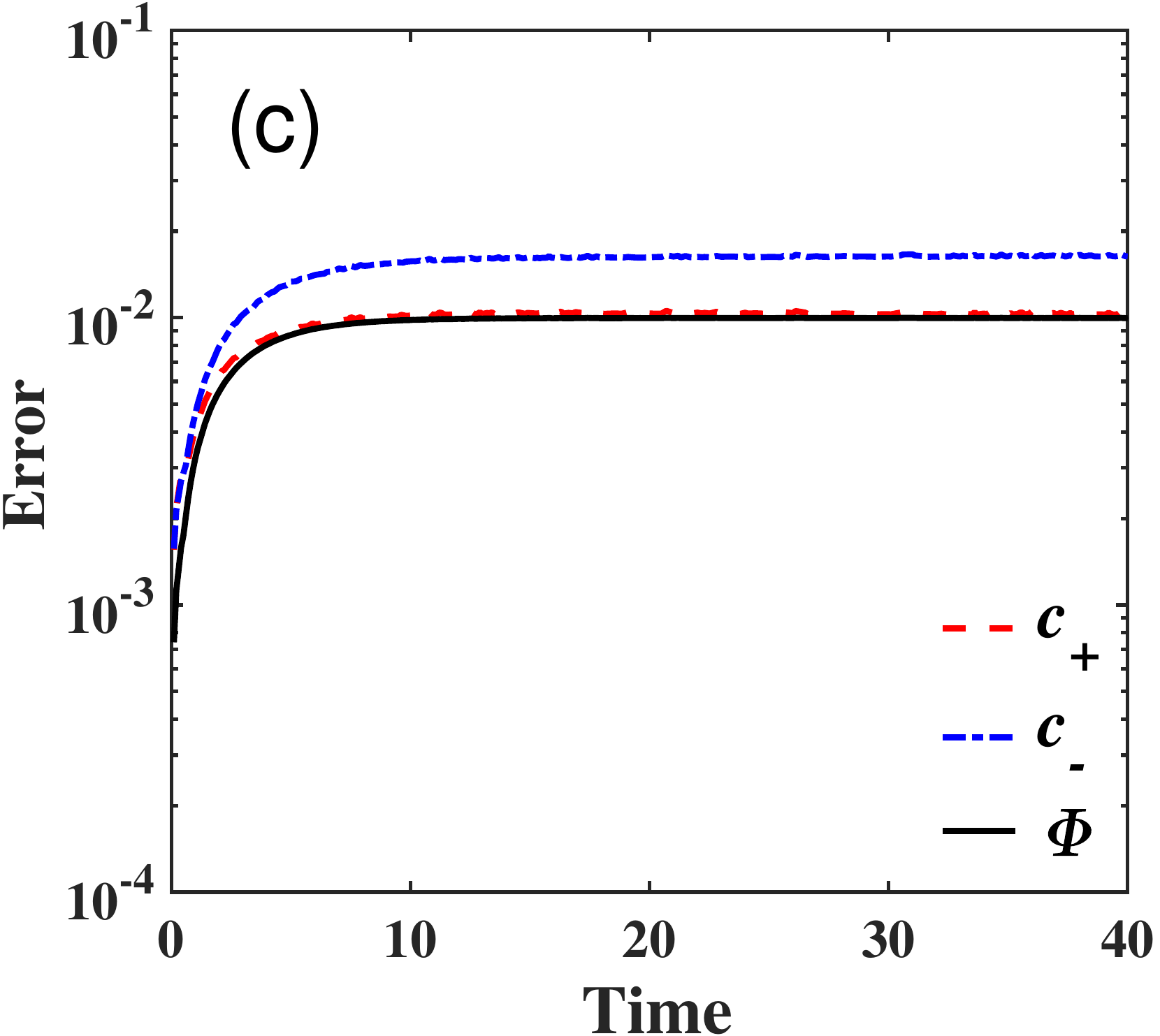}
  \includegraphics[width=0.395\textwidth]{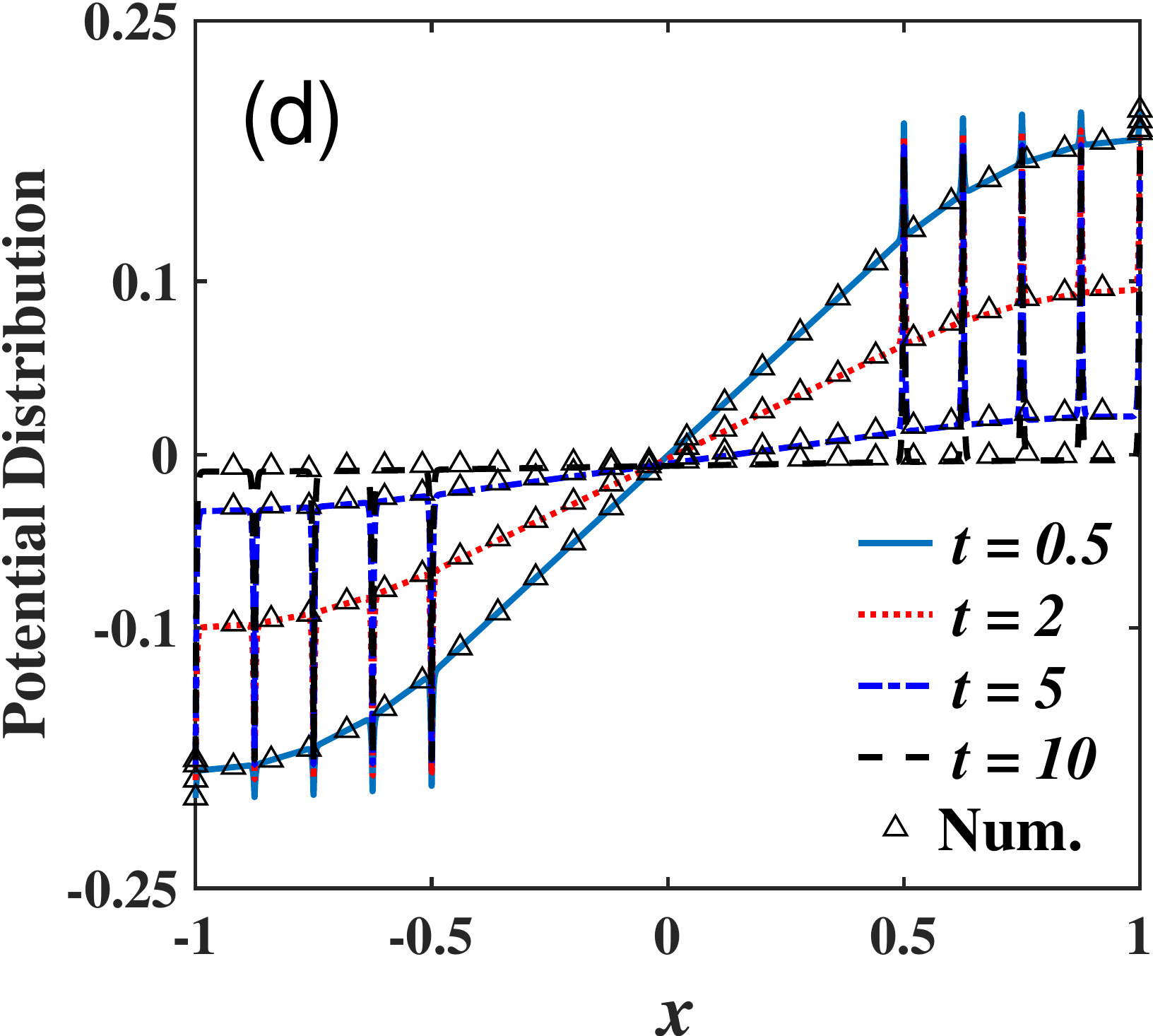}
\caption{($a$) and ($c$) Relative errors between the asymptotic solution and finite-difference solution of ionic concentrations and the potential distribution. ($b$) and ($d$) Snapshots of the electrostatic potential in the left half domain at different time $t$, obtained by the asymptotic solution (lines) and finite-difference solution (triangles). ($a$, $b$) and ($c$, $d$) are results of the symmetric and asymmetric salts, respectively. }
\label{fig:symmetric:com}
\end{figure}

For an asymmetric system, binary electrolytes with $z_+=+2$ and $z_-=-1$ is considered. Other parameters remains same as those in symmetric case. Since there is no symmetry for the left and right halves, the potential distribution and ion concentrations in the whole system is investigated. Clearly, one can see from Fig. \ref{fig:symmetric:com} (c) that the error between the leading-order asymptotic solution and finite-difference solution grows a bit in the initial stage and reaches a plateau at a small value. As displayed in Fig. \ref{fig:symmetric:com} (d), the electrostatic potential obtained with both methods has a good agreement throughout the whole charging process.  Asymmetry is faithfully captured by the leading-order asymptotic solution, and becomes more obvious at the late stage, e.g., $t=10$.

\begin{figure}[htbp]
  \centering
\includegraphics[width=0.47\textwidth]{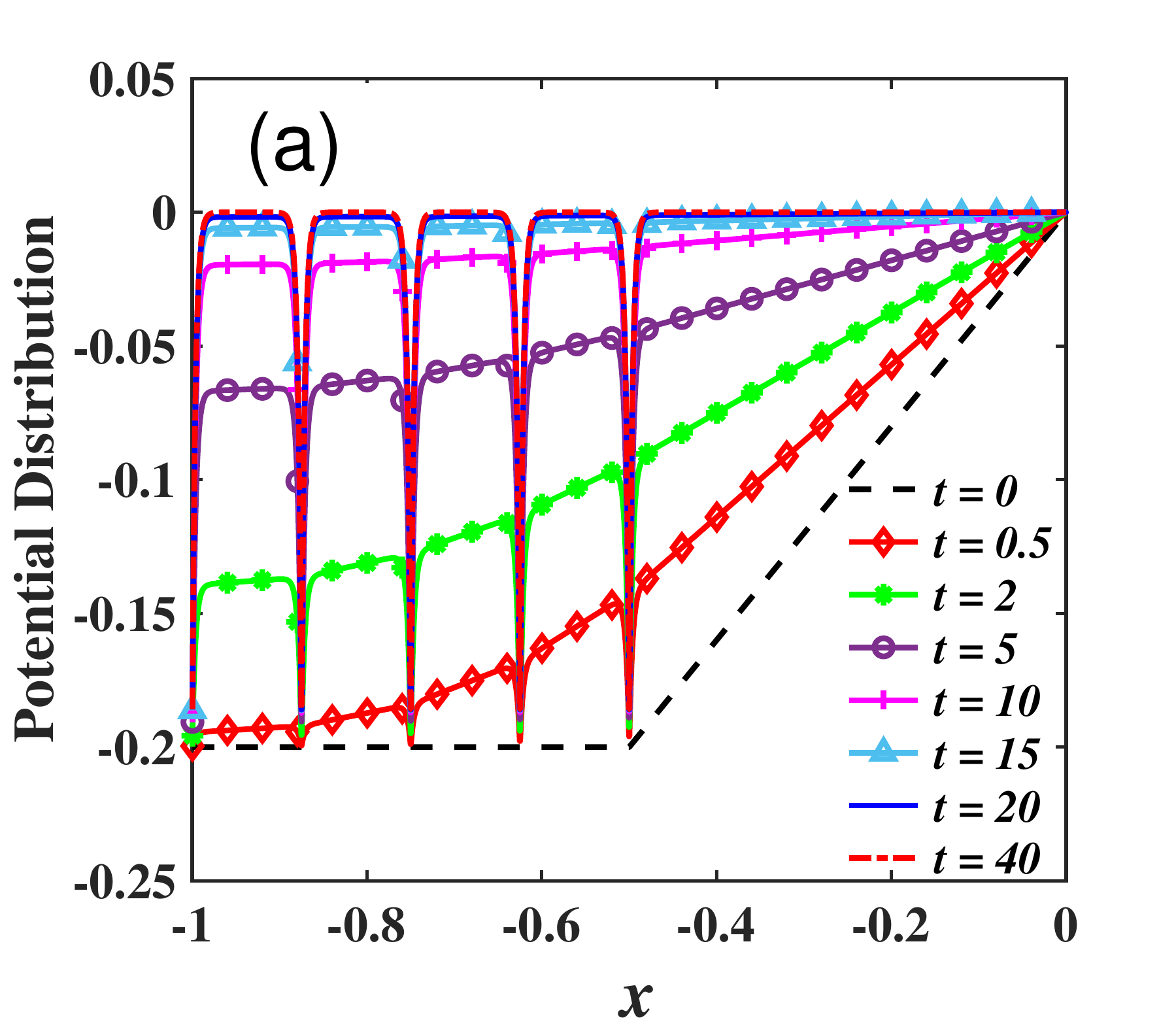}
\includegraphics[width=0.42\textwidth]{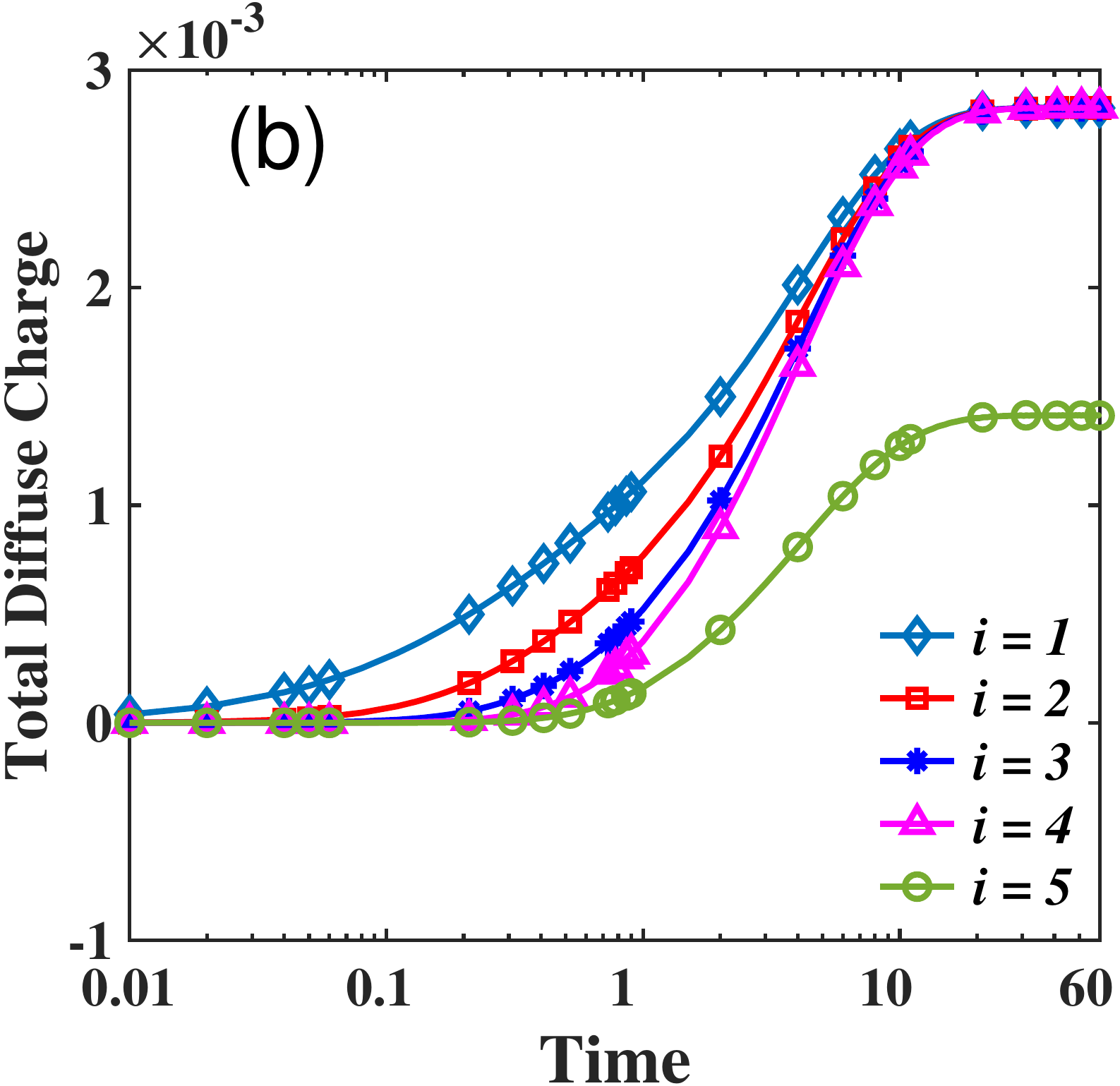}
\caption{Symmetric salt. Snapshots of potential distributions ($a$) and evolution of total diffuse charge next to the $i$th electrode ($b$), obtained by the leading-order asymptotic solution.  
}
\label{fig:symmetric}
\end{figure}

The accurate leading-order asymptotic approximation is used to study the charging dynamics by analyzing the evolution of electrostatic potential and total diffuse charge next to each electrode. For the symmetric case, the left half of the system is present in Fig.~\ref{fig:symmetric}. Panel (a) shows that the electrostatic potential is screened during the charging process and relaxes to zero in the bulk. Clearly, one can observe linear potential in the bulk at different times, which gets drastically screened by oppositely attracted ions in boundary layers, as revealed in our matched asymptotic analysis. As the system approaches equilibrium, the slope of the linear electrostatic potential between electrodes, i.e., $j_0(t)$, vanishes gradually. The zero potential in the bulk also indicates that function $A(t)$ eventually approaches zero in such a symmetric system. Panel (b) depicts the total diffuse charge next to each electrode. As charging proceeds, the diffuse charge increases quickly from zero to a uniform saturation value, except the outmost electrode ($i=5$), whose saturation value is only one half of that of the other electrodes. This is explained by the fact that the outmost electrode is only charged at one side. It is of interest to see that the electrodes closer to the center get charged up faster, because the ions in the center penetrate gradually into the stack electrodes in the charging process.

\begin{figure}[H]
  \centering
  \includegraphics[width=0.305\textwidth]{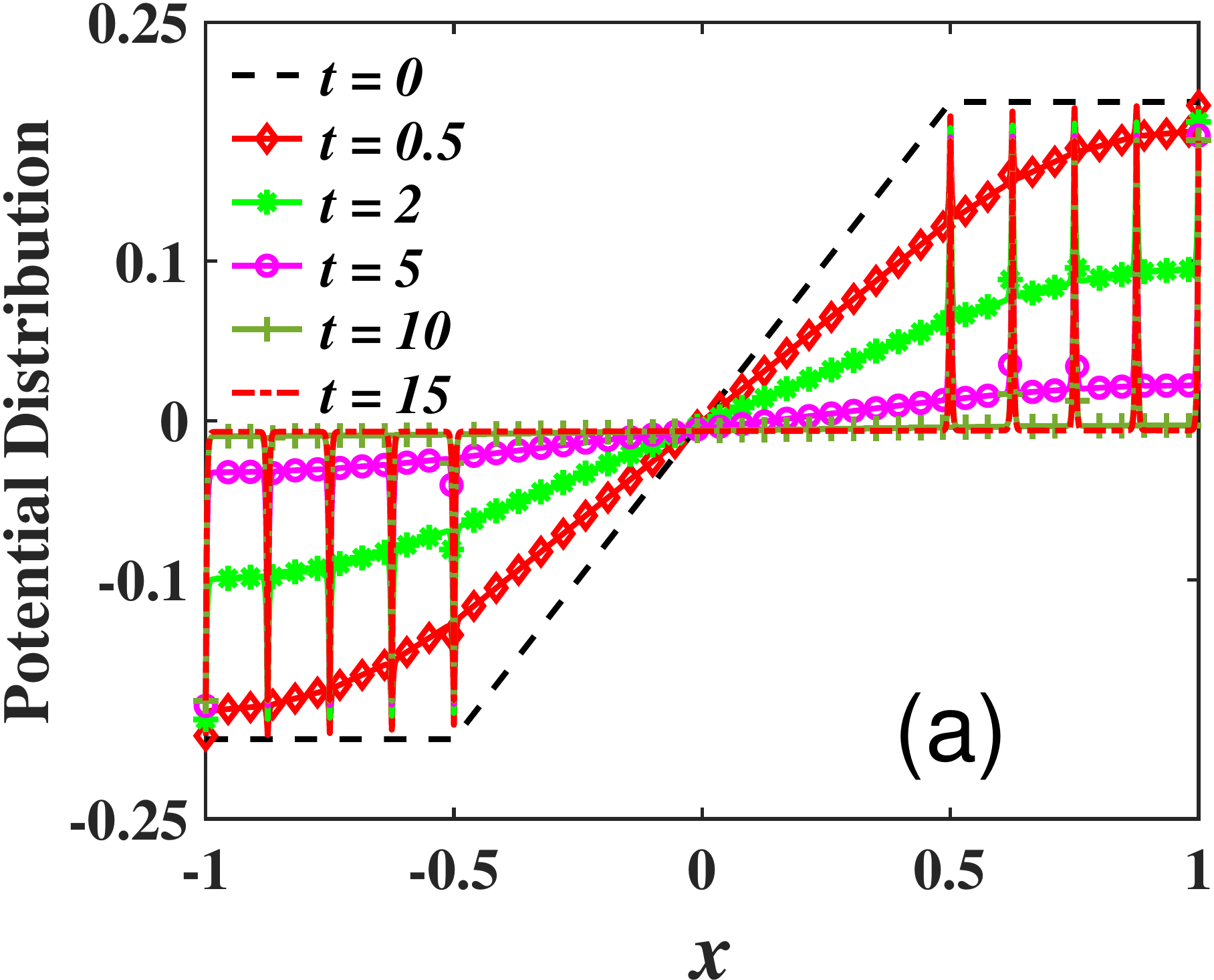}
\includegraphics[width=0.3\textwidth]{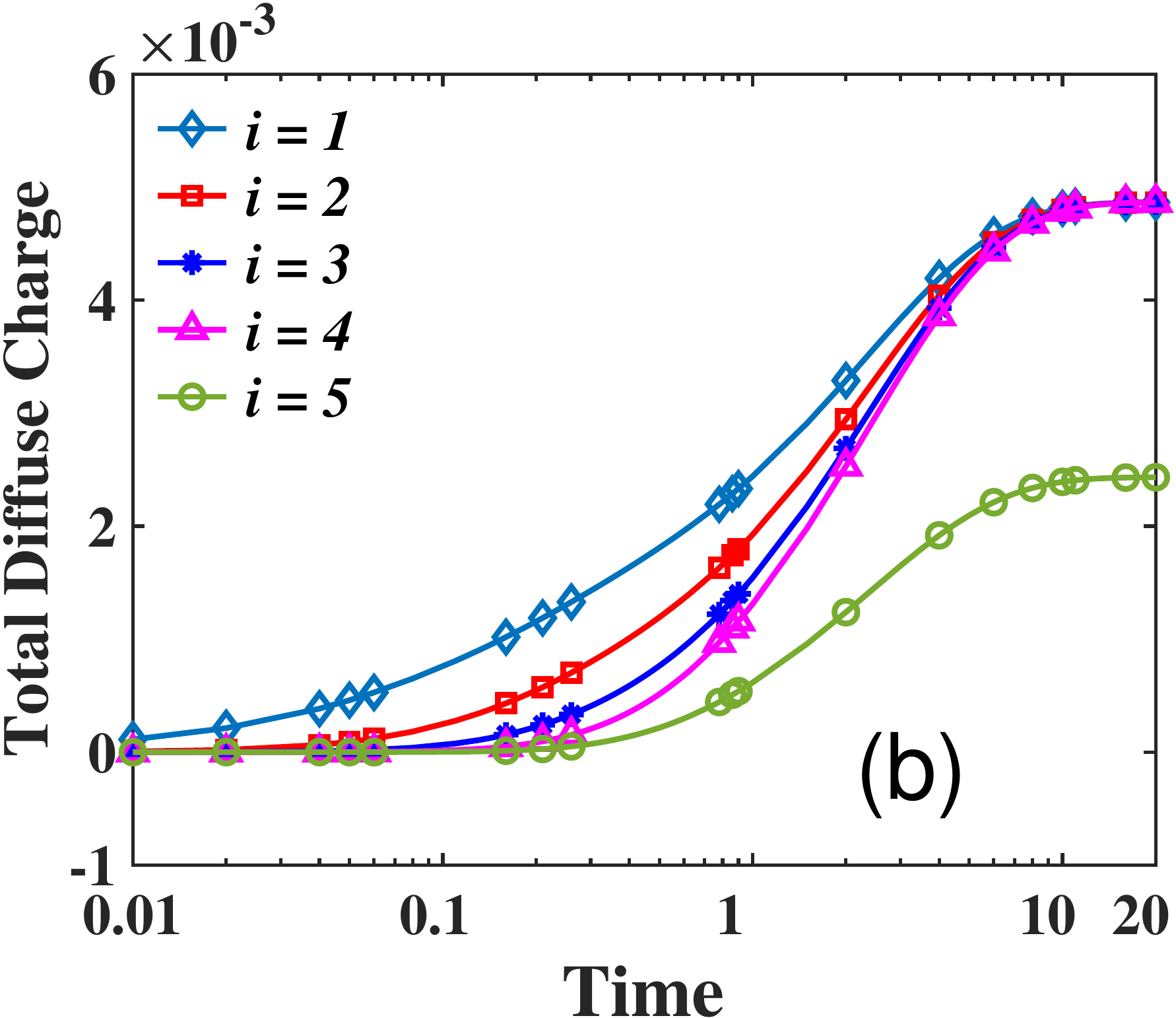}
\includegraphics[width=0.3\textwidth]{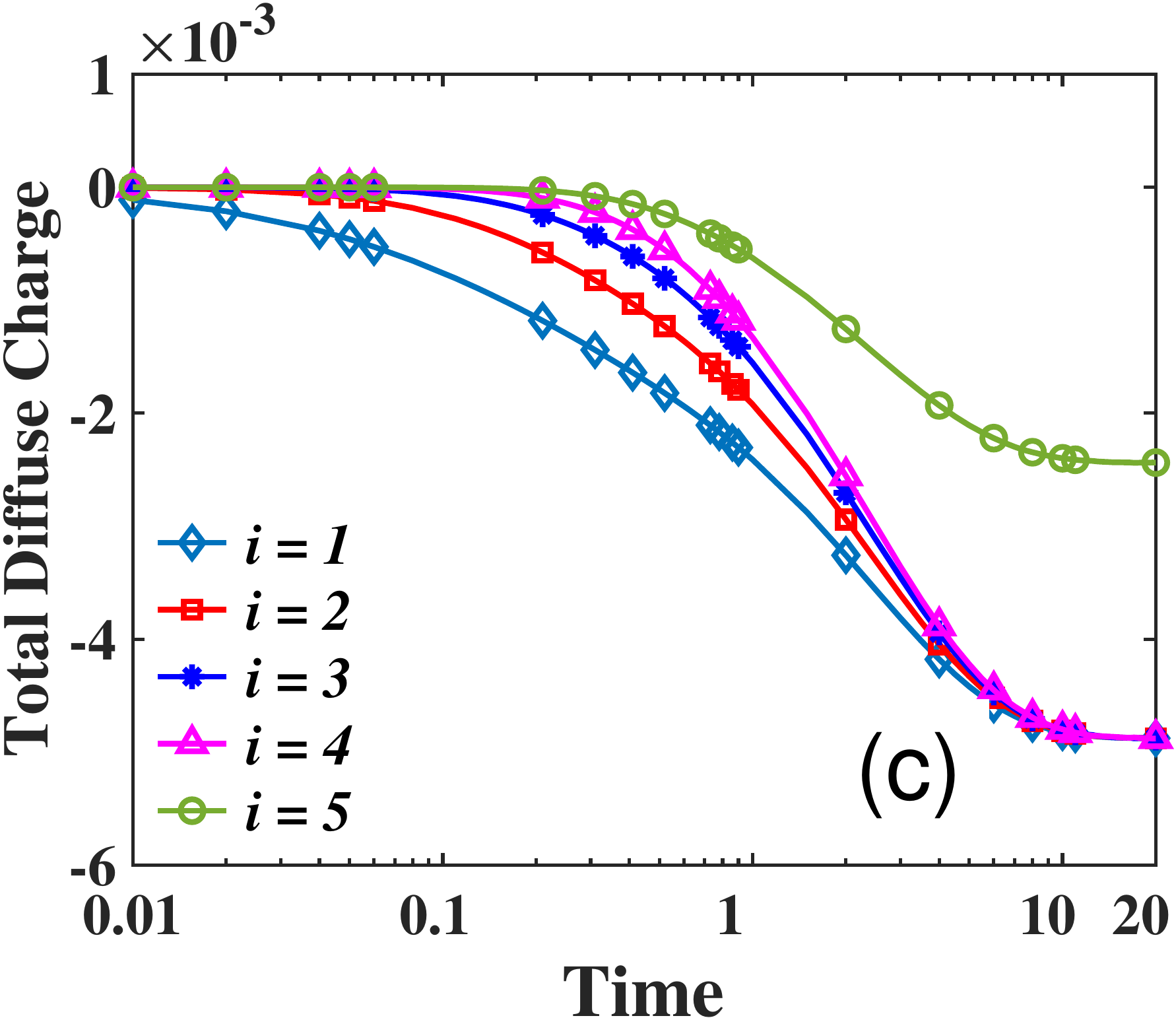}
\caption{Asymmetric salt. Snapshots of potential distributions ($a$), evolution of total diffuse charge next to the $i$th electrode in the left half ($b$) and right half ($c$), obtained by the leading-order asymptotic solution.   
}
\label{fig:asymmetric}
\end{figure}

We next consider the charging dynamics of asymmetric cases. Fig.~\ref{fig:asymmetric} displays the evolution of the electrostatic potential and total diffuse charge at the left and right half of the system. Again, one can observe that the potential keeps linear in the bulk and gets screened quickly by counterions. The slope of potential between electrodes, $j_0(t)$, also vanishes gradually as the system approaches equilibrium, whereas the function $A(t)$ that accounts for the asymmetry does not vanish eventually as that in the symmetric case. In addition, it is noticed that saturation value of total diffuse charge at electrodes becomes larger than the symmetric case, indicating that larger differential capacitance can be achieved in asymmetric cases.  Furthermore,  the asymmetric system reaches the equilibrium faster than the symmetric case due to divalent ions with higher ionic valence.  

\subsection{Generalized RC timescale}
\begin{figure}[htbp]
 \centering
\includegraphics[width=0.4\textwidth]{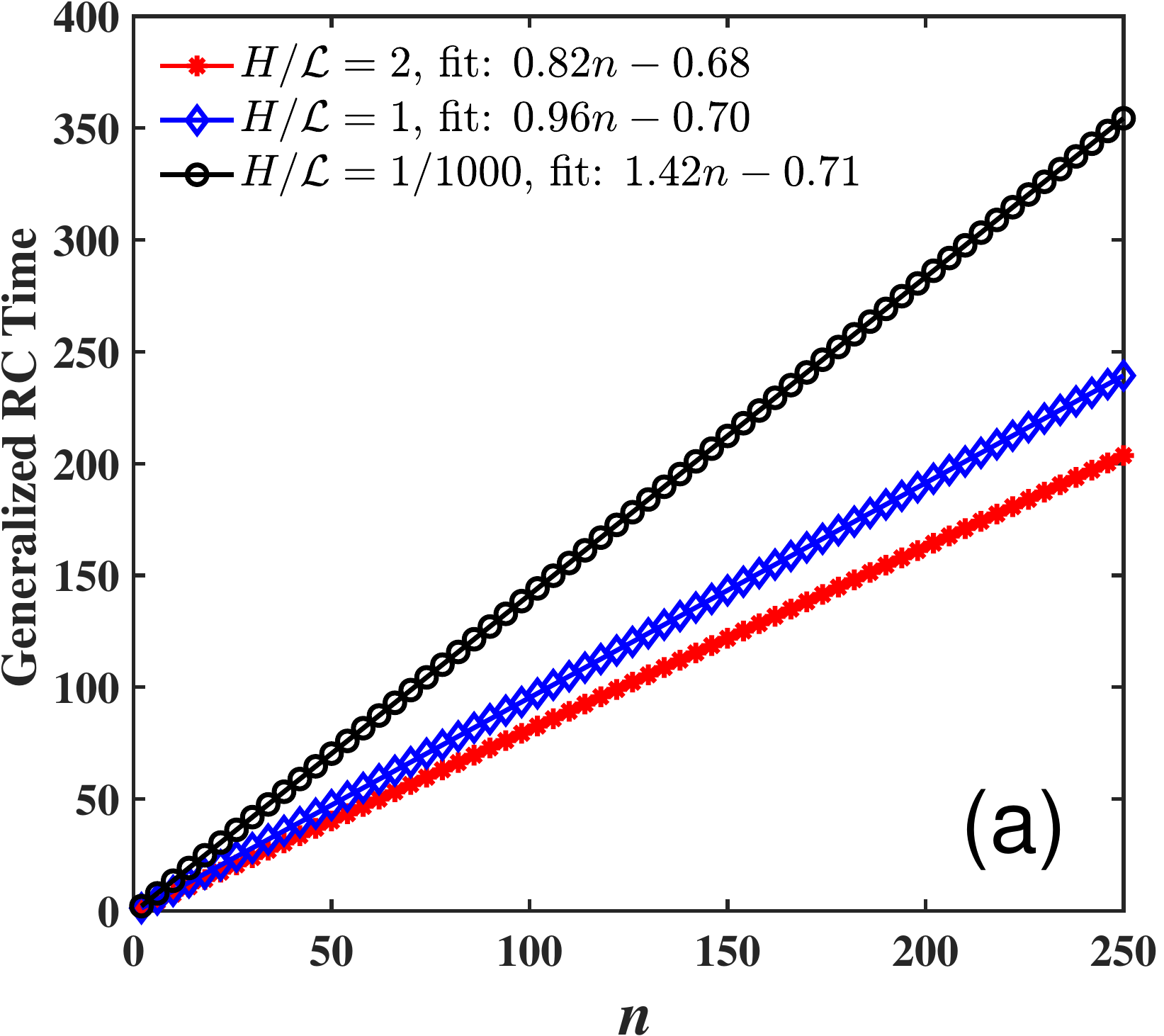}
\includegraphics[width=0.4\textwidth]{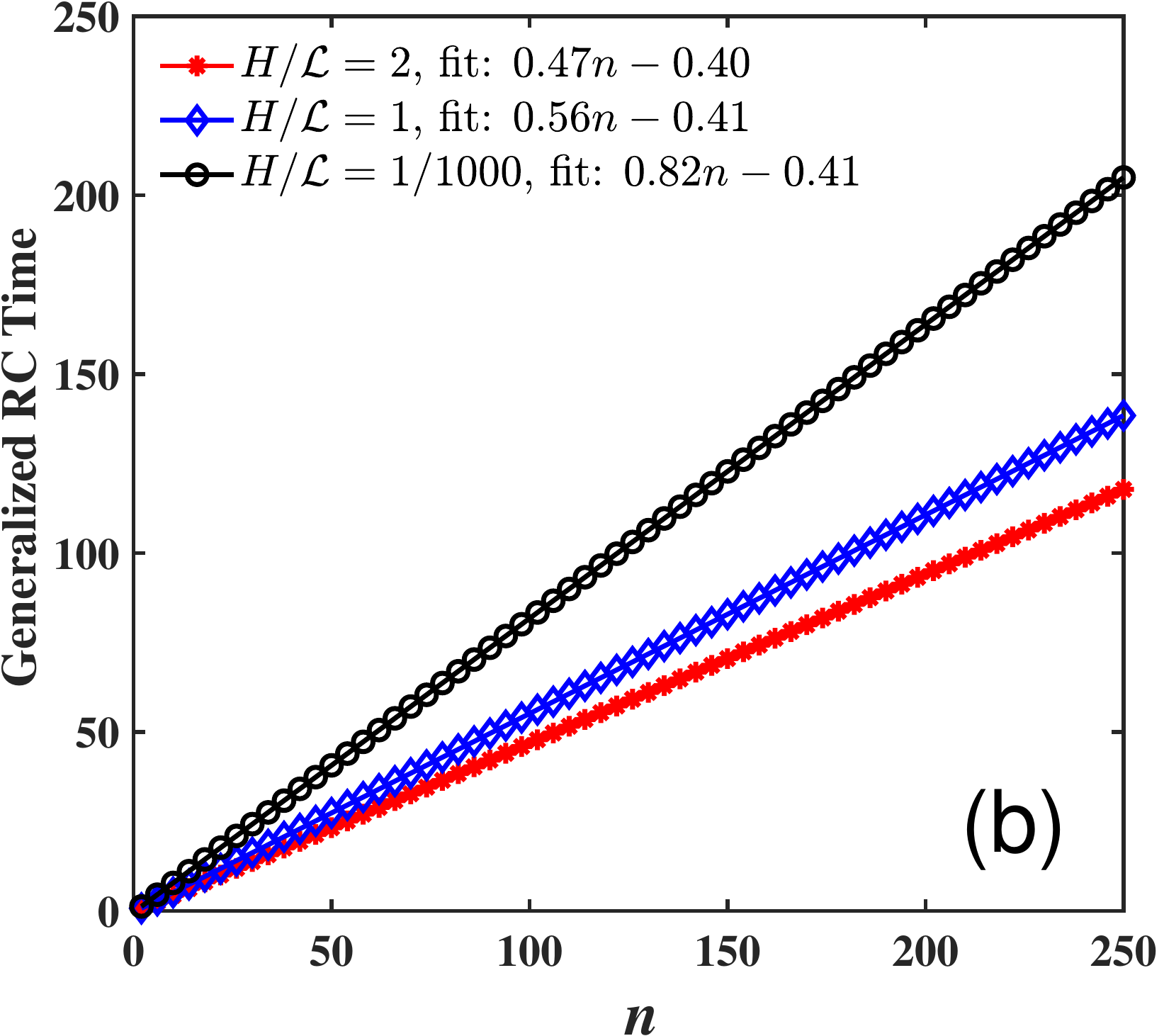}
\caption{{The charging timescale $\tau_n$ against $n$ of the $2n$ stack-electrode model for the symmetric salt (a) and the asymmetric salt (b) with different $H/\calL$.}}
\label{fig:eigenvalue}
\end{figure}

From the charging dynamics shown in Fig.~\ref{fig:symmetric}, one can know that the charging process for each electrode is different but they reach the equilibrium almost at the same time. Timescale analysis in Theorem \ref{thm:generalized RCtime} reveals that there is a generalized charging timescale $\tau_n$.
We postulate that the eigenvalues depend on the value of $n$, $h = H/(n-1)$, and $H/\calL$. We  perform a series of simulations with an applied potential $V_\pm =\pm 0.2$ and different combinations of the parameters $n$, $H$, and $H/\calL$. Consider a symmetric system with $z_+=1$ and $z_-=-1$. We numerically calculate $\lambda_c$ and show the dependence of $\tau_n$ on $n$ in Fig.~\ref{fig:eigenvalue} (a), where a clear linear dependence can be seen. By fitting, one finds that 
\begin{equation}
\tau_n = A(H) n+ B(H),
\end{equation}
where the coefficients $A$ and $B$ are functions of $H$. 
The linear dependence of $\tau_n$ on $n$ demonstrates that the multiple stack-electrode model has a generalized RC timescale that grows linearly with the number of stacks. {Our results are consistent with those in \cite{janssen2021transmission} and \cite{lian2020blessing}, where the authors have derived a formula for the generalized RC time.} Since the computational domain has been rescaled to $[-1, 1]$, smaller $H$ means a narrower electrode region. As $n$ increases, the porosity of a system with a smaller $H$ increases faster. This explains the higher slopes $A(H)$ observed in Fig.~\ref{fig:eigenvalue} (a) with a smaller $H$ value.  To mimic the porous electrodes, the number of stacks could be large and corresponding the charging timescale can be much larger.  This offers an explanation of timescale gap of several orders of magnitude between experimental data and theoretical models for supercapacitors described by two parallel plates.

For the asymmetric case, a series of simulations with the same parameters except $z_+=2$ and $z_-=-1$ are performed. The generalized charging timescale is again calculated by $\tau_n=1/\lambda_c$. Fig.~\ref{fig:eigenvalue} (b) shows that the charging timescale $\tau_n$ again scales linearly with the number of stacks. 
In contrast to the symmetric case,  one can find that the slope of the linear functions is smaller than that of the symmetric case. This further confirms that the charging timescale of the system is smaller with an asymmetric electrolyte.

\subsection{Diffusion timescale}

\begin{figure}[H]
  \centering
\includegraphics[scale=0.42]
{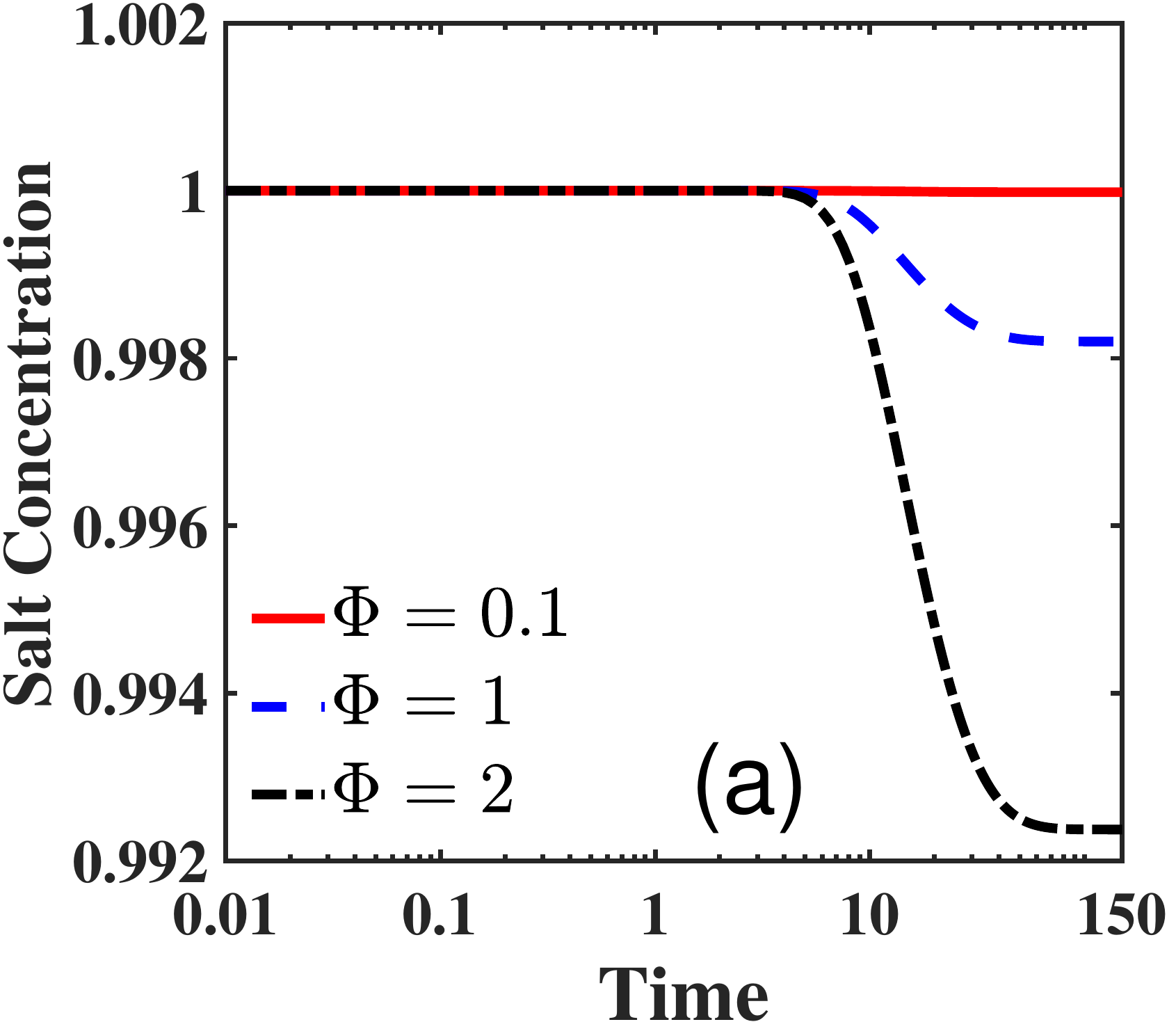}
\includegraphics[scale=0.415]{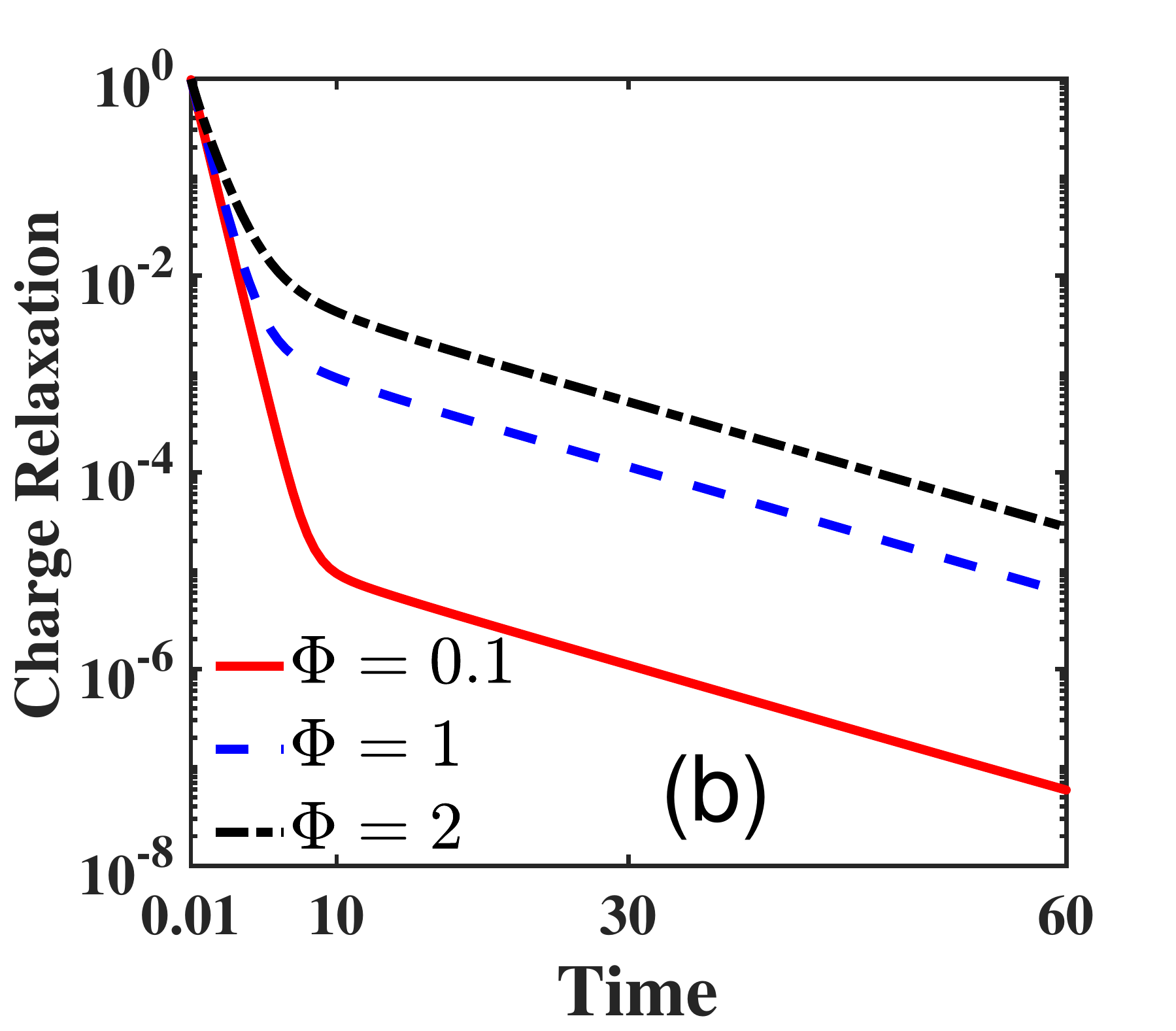}\\
\hspace{-0.5cm}
\includegraphics[scale=0.42]{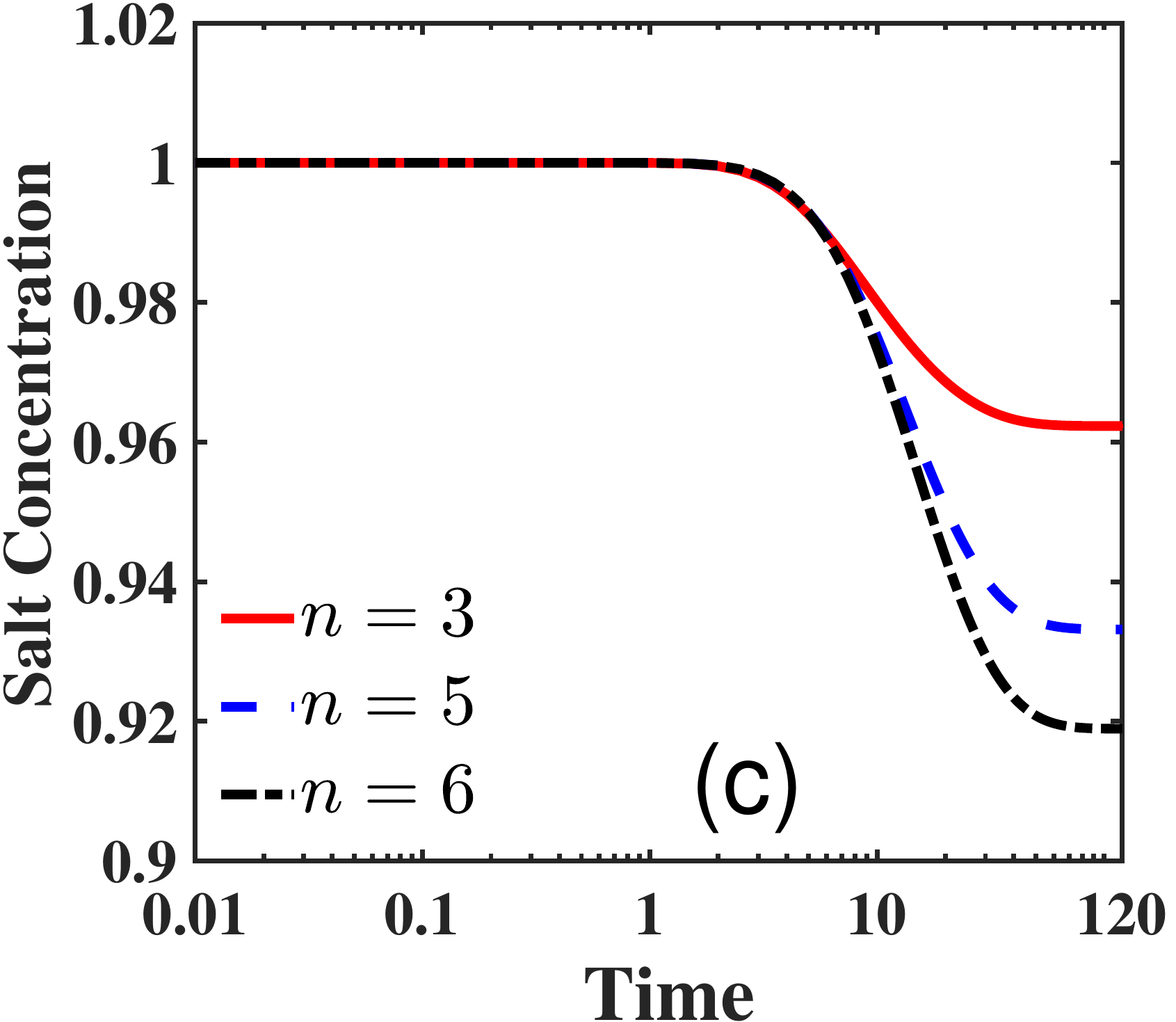}
\includegraphics[scale=0.42]{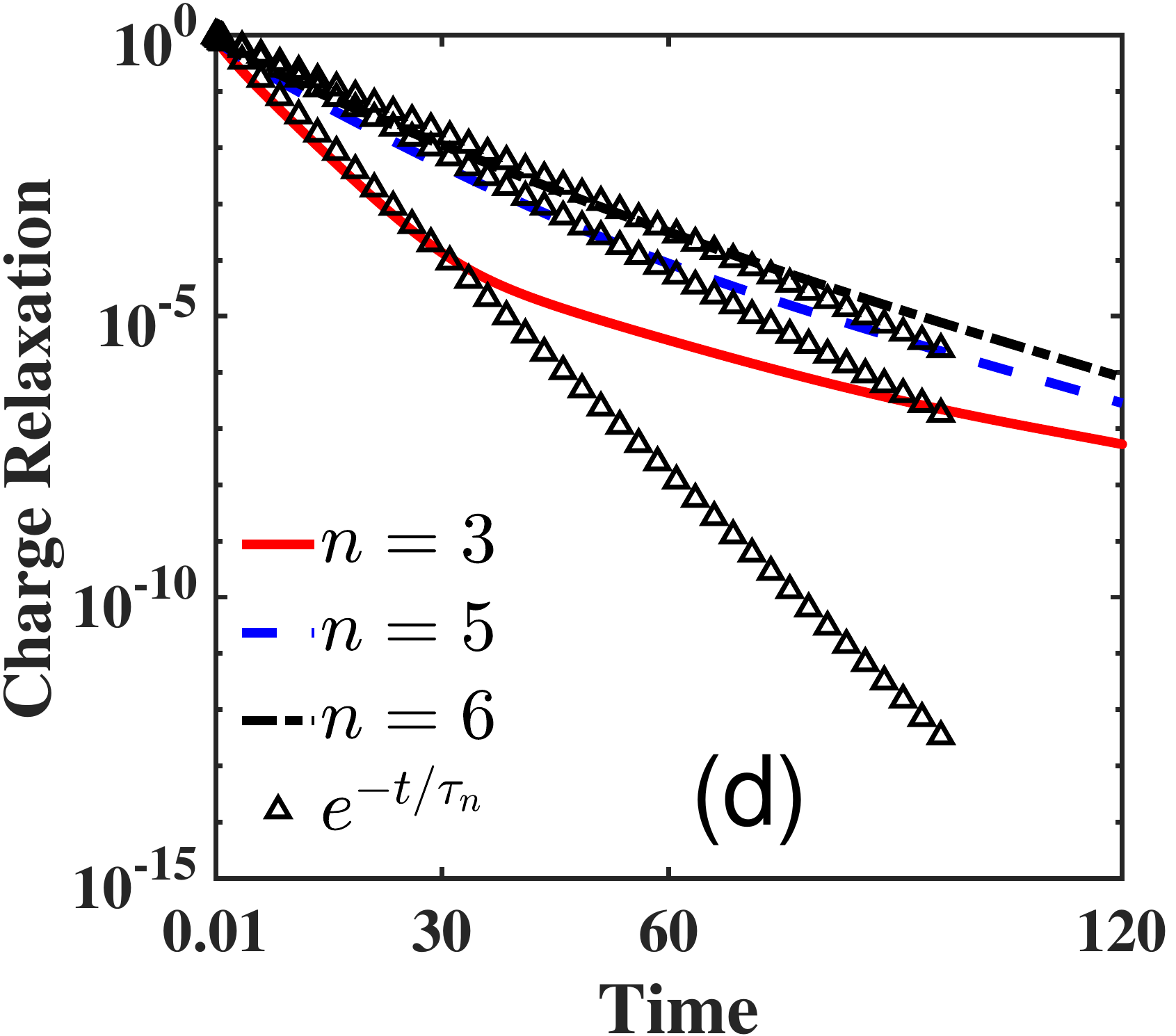}
\caption{(a) The salt concentration at the center for $n=1$; (b) The log-plot of the charge relaxation $1-Q/Q_{eq}$ under different applied potentials: $V_\pm =\pm 0.1, \pm 1, \pm 2$ for $n=1$;  (c) The salt concentration at the center for $n=\{3,5,6\}$; (d) The log-plot of the charge relaxation with $V_\pm =\pm 2$ for $n=\{3,5,6\}$.}
\label{fig:symmetric:diffuse}
\end{figure}
Analysis in Section~\ref{sec:diffusiontime} indicates that, for large applied voltages, there exists a longer diffusion timescale $\overline{\tau}=\tau_c/\epsilon = D^2/\mathcal{D}_0$, which is in dependent of $n$ but depends on the system size $D$. Here, we carry out simulations to demonstrate the existence of two timescales: diffusion timescale $\overline{\tau}$ and the generalized RC charging timescale $\tau_n$ studied above.

Consider a symmetric system ($z_\pm =\pm 1$) first with $n=1$, i.e., $H=0$ and $\mathcal{L}=1$.
The charge relaxation $1-Q/Q_{eq}$ close to the electrode at $x=-\mathcal{L}$ and the salt concentration $c =(z_+c_+(0,t)-z_-c_-(0,t))/(z_+ - z_-)$ at $x=0$ are investigated with three different applied potentials: $V_\pm =\pm 0.1, \pm 1, \pm 2$. It is of interest to observe from Fig.~\ref{fig:symmetric:diffuse} (a) that, only under large applied voltages, the salt concentration in the bulk can be much depleted in the second phase of the charging process.  The curves of charge relaxation in the logarithmic scale in Fig.~\ref{fig:symmetric:diffuse} (b) indicate that there are two distinct relaxation phases with two separate timescales, which have been investigated in~\cite{BTA:PRE:04}. The first one is the generalized RC charging timescale and the second one is the diffusion timescale.  Fig.~\ref{fig:symmetric:diffuse} (c) displays salt concentration at center for $n=\{3,5,6\}$ with $V_\pm =\pm 2$, $H =0.5$, and $L=0.5$. Two distinct relaxation phases are again observed. In the last plot of Fig. \ref{fig:symmetric:diffuse}, the reference curves $\text{exp}(-t/\tau_n )$ with $n=3$, $5$, $6$, are display, where $\tau_n$ is predicted by the asymptotic analysis. Obviously, one can find that the reference curves are consistent with the first phase of the curves for the total diffuse charge, indicating that the generalized RC timescale $\tau_n$ found by our asymptotic analysis works well in the first phase of charging. As $n$ increases, it is seen that the two timescales tend to be the same, being consistent with~\cite{lian2020blessing}.

\section{Conclusion}
\label{sec:conclusion}

In this work, asymptotic analysis on the ion concentrations, potential distributions, and total diffuse charge has been performed for stack-electrode supercapactiors. 
The asymptotic analysis derives a generalized equivalent circuit model for the zeta potentials at all stacks, which takes potential-dependent nonlinear capacitance and resistance determined by the PNP theory and physical parameters of electrolytes, e.g., specific counterion valences for asymmetric electrolytes.
Resorting to the developed linearized stability analysis, projection technique, and asymptotic analysis on excess salt concentration, the biexponential charging timescales have been analytically studied for two-ion system.
The validity of the asymptotic solution has been verified by a series of numerical tests. Further numerical investigations on the charging timescale have demonstrated that our generalized equivalent circuit model, as well as companion linearized stability analysis, can faithfully capture the charging dynamics of symmetric/asymmetric electrolytes in supercapacitors with porous electrodes.

The MAE method for stack-electrode supercapactiors can be possibly extended to derive reduced models for the description of ionic correlation effects~\cite{BSK:PRL:2011}, reversible and irreversible heat generation in porous electrodes~\cite{JiLiuLiuZhou_JPS22}, and large capacitance of double layers formed at metal electrodes due to image-charge attractions~\cite{SLS:PRL:2010, ji2018asymptotic}. Linearized stability analysis on reduced models could elucidate the impact of structure porosity on the charging dynamics in supercapactiors with porous electrodes.

\section*{Data accessibility}

Software code to reproduce the data in the paper is available at \url{https://github.com/LijieJi6/AsymptoticAnaStackElectro}.

\section*{Authors' contributions}
Lijie Ji: Conceptualization, Methodology, Formal analysis, Software, Writing - Original Draft, Writing - Review and Editing, Funding acquisition.
Zhenli Xu: Conceptualization, Methodology, Formal analysis, Supervision, Writing - Original Draft, Writing - Review and Editing, Funding acquisition, Project administration.
Shenggao Zhou: Conceptualization, Methodology, Formal analysis, Supervision, Writing - Original Draft, Writing - Review and Editing, Funding acquisition, Project administration.
All authors have read and approved the manuscript. All authors agree to be accountable for all aspects of the work.

\section*{Competing interests}
We declare we have no competing interests.

\section*{Funding}
The authors are funded by China Postdoctoral Science Foundation (No. 2021M702141), the NSFC (grant Nos. 12201403, 12071288 and 12171319), the Science and Technology Commission of Shanghai Municipality (grant Nos. 20JC1414100 and 21JC1403700), and the Strategic Priority Research Program of CAS (grant No. XDA25010403). This work is also supported by the HPC center of Shanghai Jiao Tong University.

\bibliographystyle{abbrv}
\bibliography{MultiElectrode}

\end{document}